\documentclass[12pt]{amsart}
\usepackage{amssymb, latexsym, amsthm,amsmath,amscd,hyperref,MnSymbol,graphicx}
\DeclareGraphicsExtensions{.png}
\theoremstyle{plain}
\newtheorem{thm}{Theorem}[section]

\newtheorem{prop}[thm]{Proposition}

\theoremstyle{definition}
\newtheorem{rem}[thm]{Remark}

\newtheorem{exmpl}[thm]{Example}

\newtheorem{cor}[thm]{Corollary}
\newtheorem{defn}[thm]{Definition}
\newtheorem{con}[thm]{Construction}
\newtheorem{note}[thm]{Notation}

\newif\ifXY 
\XYtrue     

\ifXY
\usepackage{xy}
\fi \ifXY \xyoption{matrix} \xyoption{arrow} \xyoption{curve} \fi

\address{}
\email{tendlea@gmail.com}
\thanks{}

\newcommand{\form}[1]{(\ref{Eq:#1})}
\newcommand{\re}[1]{\ref{E:#1}}

\newcommand{\rs}[1]{Section \ref{S:#1}}
\newcommand{\rss}[1]{Subsection \ref{SS:#1}}

\begin{document}

\begin{titlepage}
\begin{center}

\begin{figure}
\includegraphics[width=0.15\textwidth]{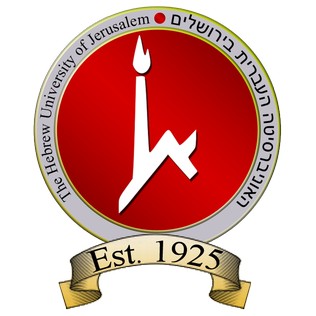}~\\[0.5cm]
\end{figure}

\textsc{\Large The Hebrew University of Jerusalem}\\[0.5cm]
\textsc{\Large Faculty of Science}\\[0.5cm]
\textsc{\Large Einstein Institute of Mathematics}\\[1.5cm]

\textsc{\large Master thesis on the subject of:}\\[1.5cm]

{ \huge \bfseries Geometric Class Field Theory \\[1.2cm] }


\noindent
\begin{minipage}[t]{0.4\textwidth}
\begin{flushleft} \large
\emph{Author:}\\
Avichai \textsc{Tendler}
\end{flushleft}
\end{minipage}%
\begin{minipage}[t]{0.4\textwidth}
\begin{flushright} \large
\emph{Supervisor:} \\
Prof.~Yakov \textsc{Varshavsky}
\end{flushright}
\end{minipage}

\vfill

{\large August 5, 2010}

\end{center}
\end{titlepage}

\title{Geometric class field theory}

\begin{abstract}
In this paper we prove global class field theory using a purely geometric result. We first write in detail Deligne's proof to the unramified case of class field theory, including defining the required objects for the proof. Then we generalize the notions appearing in the proof to prove also the tamely ramified case relying on the unramified one.
\end{abstract}

\maketitle

\section{Introduction}\label{S:INT}

In this paper, we assume familiarity with some standard notions from commutative algebra, algebraic geometry, algebraic number theory and the theory of \'etale morphisms, which can be found for example in \cite{AM}, \cite{Ha}, \cite{CF} and \cite{Mi} respectively.

Class field theory was developed in the period between 1850 to 1950. Its results describe the Galois groups of abelian extensions of global and local fields.

One type of global fields, which will be the only type we deal with in this work, are function fields over finite fields. These are field extensions of transcendence degree 1 over the finite field $\mathbb{F}_q$. Function fields enjoy the nice properties of global fields, in particular the Galois group of an abelian extension can be described in terms of the valuations of the field. But in addition, they have unique special advantages of their own, e.g., there are no Archimedean places in function fields. More importantly, function fields can be realized, as their name suggests, as fields of rational functions on algebraic curves.

The first proofs of class field theory didn't make use of these advantages of the function field case and are somewhat complicated. In the 50s, S. Lang, M. Rosenlicht and J. -P. Serre proved class field theory using algebraic geometry and the close relation between function fields and algebraic curves. We develop this connection in \rs{GPG} and \rs{EFG}. The old meaning of the term "geometric class field theory" is this theory of Lang and Rosenlicht, which can be found in \cite{Se1}.

Later, P. Deligne gave an elegant proof of unramified global class field theory over function fields using $\ell$-adic sheaves. In modern language, the notion "geometric class field theory" refers to working with sheaves rather than functions, and this is the language we use starting from \rs{GSFC}. In this modern viewpoint we get the result of class field theory as a corollary of the main theorem \re{MTRP}. A sketch of the proof of the main theorem, only in its unramified version, can be found also in \cite{La} (only the complex case) or in \cite{Fr1} (the general base field case).

In addition to its beauty, Deligne's proof has a more important advantage, instead of describing $Gal(K^{ab}/K)$, the Galois group of the maximal abelian extension of the field $K$, it describes characters of the Galois group of the maximal extension: $Gal(\bar{K}/K)\rightarrow\bar{\mathbb{Q}}^{\times}_{\ell}$. Although it is equivalent, characters can be generalized to representations in non-abelian groups, such as $Gl_n(\bar{\mathbb{Q}}_{\ell})$ or more generally, a reductive group $G$ (conjecturally), in order to obtain a non-abelian generalization of class field theory. This is the essence of the Langlands program, but we shall not deal with the non-abelian case here.

In this paper we first define the required notions and state the unramified and ramified versions of global class field theory (theorems \re{UCFT} and \re{RCFT} respectively). This is done in \rs{CGCFT}.

We develop the geometric objects which correspond to the ones occurring in class field theory. The most important objects here are the generalized Picard group (\rs{GPG}) and the \'etale fundamental group of a scheme (\rs{EFG}); these objects correspond to the adelic side and the Galois side in classical class field theory, respectively. These sections also explain how the classical and geometric objects are connected. In fact,  this is a restatement of the classical theorems of class field theory in the language of algebraic geometry.

In \rs{GSFC} we explain relations between sheaves and functions. The motivation here is to consider things from the modern viewpoint of algebraic geometry, which uses sheaves rather than functions. As we shall see, this viewpoint is more convenient to work with since standard operations on functions have analogous operations on sheaves, but there are additional ways to manipulate sheaves that don't exist for functions.

\rs{PFPCFT} finally starts to prove class field theory. Class field theory is a correspondence between adelic and Galois representations. In this section we prove the easier part of the correspondence, we construct from an adelic representation a Galois one, saving the converse construction for later. We prove here the unramified and tamely ramified cases.

\rs{TMT} gives the statement of the geometric result from which class field theory follows. There are three versions of the theorem: for the unramified case, the ramified case, and the ramified version in characteristic zero. The third has nothing to do with classical class field theory, but we state it for its own interest. In this section we also show why the geometric results imply class field theory.

\rs{PMTU} is devoted to a detailed explanation of the proof of the geometric theorem in its unramified case. While \rs{PMTR} gives the modifications of the proof in the ramified version, the characteristic zero case and the tamely ramified case follows relatively easily from the unramified case after all this preliminary work.

\section{classical global class field theory}\label{S:CGCFT}

In this section, we state some of the results of global class field theory in the function field case. Proofs will be given later using geometric techniques.

\begin{note}
We use some standard notations from class field theory: $K$ will be a function field over $\mathbb{F}_q$, $q=p^n$, $p\in\mathbb{Z}$ is a prime number and $\ell\neq p$ is another prime. We let $v$ or $\mathfrak{p}$ denote a valuation of $K$ . By abuse of language we will sometimes interchange between the words prime, valuation and place.  $K_v$ (or $K_{\mathfrak{p}}$) will be the completion of $K$ at $v$.
$\mathcal{O}_v$ (or $\mathcal{O}_{\mathfrak{p}}$) the ring of integers in $K_v$ and $\mathfrak{m}_v$ the corresponding maximal ideal
of $\mathcal{O}_v$.

$\bar{K}$ is the separable closure of $K$ and $G_K$ will denote $Gal(\bar{K}/K)$ the absolute Galois group of $K$.
\end{note}

\begin{defn}[Adele ring]
The {\em adele ring} of $K$ is the restricted topological product $\mathbb{A}_K:=\underset{v}{\prod}'K_v$ with respect to
the open subrings $\mathcal{O}_v$, the product is taken over all the primes of $K$.

According to the definition of restricted topological product, the group of invertible adeles is
\begin{center}
$\mathbb{A}_K^{\times}=\{{(a_v)}_v\in \underset{v}{\prod} K_v^{\times}: a_v\in \mathcal{O}_v^{\times}\, for \, almost \, all \, v\}$
\end{center}
There exists a subring $\mathbb{O}_K=\underset{v}{\prod}\mathcal{O}_v$ of $\mathbb{A}_K$.
\end{defn}

There exists also an inclusion $i:K\hookrightarrow\mathbb{A}_K$ defined by $a\mapsto{(a)}_v$ which has a discrete image.
By abuse of notation we denote the image $i(K)$ by $K$.

The following definition will help us to establish a connection between extensions of a global field and its localizations:

\begin{defn}
\begin{enumerate}
\item
Let $K'/K$ be a finite Galois extension with Galois group $G_{K'/K}$, $\mathfrak{p}$ a prime ideal of $K$ and $\mathfrak{P}$ is a prime over $\mathfrak{p}$ in $K'$.

The {\em decomposition group} is defined to be:
\begin{center}
$D(\mathfrak{P})=\{\sigma\in G_{K'/K}: \sigma \mathfrak{P}=\mathfrak{P}\}$
\end{center}

\item
A {\em compatible set of primes} is a choice of a prime $\mathfrak{P}'$ over $\mathfrak{p}$ for any finite Galois extension $K'\supseteq K$. Such that if $K''\supseteq K' \supseteq K$, and $\mathfrak{P}'',\mathfrak{P}'$ are the corresponding primes then $\mathfrak{P}''$ is over $\mathfrak{P}'$. In this case, for such extensions we have a map $D(\mathfrak{P}'')\rightarrow D(\mathfrak{P}')$.

\item
For a compatible system primes $\mathfrak{P}$ over $\mathfrak{p}$, the inverse limit of the decomposition groups  $D(\underset{K\subseteq K'}{\varprojlim}\mathfrak{P}')$ can be defined.
\end{enumerate}
\end{defn}

\begin{rem}\label{E:Inertia}\
\begin{enumerate}
\item
Since the elements of $D(\mathfrak{P})$ acts continuously with respect to the $\mathfrak{P}$-adic topology, we obtain a map $i_1:D(\mathfrak{P})\xrightarrow{\approx} Gal(K'_{\mathfrak{P}}/K_{\mathfrak{p}})$, this map is an isomorphism. Also, $D(\underset{K\subseteq K'}{\varprojlim}\mathfrak{P}')$ is a subgroup of $G_K$ isomorphic to $G_{K_{\mathfrak{p}}}$.
\item
There exists a surjective map $i_2:Gal(K'_{\mathfrak{P}}/K_{\mathfrak{p}})\twoheadrightarrow Gal(\mathbb{F}_{q'}/\mathbb{F}_q)$ here, $\mathbb{F}_{q'}, \mathbb{F}_q$ are the residue fields of $K',K$ respectively. The kernel of the map $i_2\circ i_1$ is a normal subgroup of the decomposition group $D(\mathfrak{P})$, it is called the {\em inertia group} and it is denoted $I(\mathfrak{P})$.
\begin{center}
$1\rightarrow I(\mathfrak{P}) \rightarrow D(\mathfrak{P}) \rightarrow Gal(\mathbb{F}_{q'}/\mathbb{F}_q)\rightarrow 1$
\end{center}
\item
Note that if $\mathfrak{P}, \mathfrak{P}'$ are different primes of $K'$ over $\mathfrak{p}$, then $D(\mathfrak{P})$ and $D(\mathfrak{P}')$ are conjugate subgroups of $G_{K'/K}$. Hence a character of $G_{K'/K}$ vanishes on one iff it vanishes on the other, those the notion of vanishing on the decomposition group (or some specific subgroup of it) is independent of the choice of $\mathfrak{P}$.
\item
A similar result is true for the inverse limits: the decomposition groups for different choices of compatible system of primes $D(\underset{K\subseteq K'}{\varprojlim}\mathfrak{P})$ and $D(\underset{K\subseteq K'}{\varprojlim}\mathfrak{P}')$ are conjugate in $G_K$. A similar result applies to the inertia group.
\end{enumerate}
\end{rem}

According to the last remark, the following is well defined:

\begin{defn}
Let $\rho:G_K\rightarrow\bar{\mathbb{Q}}_{\ell}^{\times}$ be a character.
\begin{enumerate}
\item
$\rho$ is {\em unramified at $\mathfrak{p}$} if it is vanishes on the inertia group $I(\underset{K\subseteq K'}{\varprojlim}\mathfrak{P}')$ for some (hence any) compatible set of primes over $\mathfrak{p}$.
\item
$\rho$ is {\em unramified} if it is unramified at any $\mathfrak{p}$.
\end{enumerate}
\end{defn}

We shall now define special elements of the Galois group and the invertible adele group, that will correspond to one another under global class field theory.

\begin{defn}
\begin{enumerate}
Assume $\mathfrak{p}$ is a prime of $K$.

\item
Let $\pi_{\mathfrak{p}}=(a_{\mathfrak{p}'})_{\mathfrak{p}'}\in\mathbb{A}_K{\times}$ denote an element such that $a_{\mathfrak{p}'}=1$ for $\mathfrak{p}'\neq \mathfrak{p}$ and $a_\mathfrak{p}=\pi_\mathfrak{p}$ a uniformizer, i.e. a
generator of $\mathfrak{m}_{\mathfrak{p}}$.

\item
In the notation of the above definition, since $\mathbb{F}_{q'}/\mathbb{F}_q$ is an extension of finite fields, there exists the Frobenius automorphism $F\in Gal(\mathbb{F}_{q'}/\mathbb{F}_q)$, defined by $F(x)=x^q$.
Recall the exact sequence from remark \re{Inertia} (2), the inverse image $Fr_{\mathfrak{P}/\mathfrak{p}}$ of $F$ in $D(\mathfrak{P})$ is defined up to an element of the inertia subgroup $I(\mathfrak{P})$.

Note that the different decomposition groups for different primes $\mathfrak{P}$ are conjugate to one another in $Gal(K'/K)$. Therefore, $Fr_{\mathfrak{P}/\mathfrak{p}}$ defines up to inertia and conjugacy an element of $Gal(K'/K)$,  which is the {\em Frobenius element}.

\item
Taking the inverse limit of the Frobenius elements in a compatible set of primes, we obtain the Frobenius element $Fr_{\mathfrak{p}}\in G_K$, which is also defined up to conjugacy and inertia.
\end{enumerate}
\end{defn}

\begin{rem}
Note that if a character $\rho:G_K\rightarrow\bar{\mathbb{Q}}_{\ell}^{\times}$ is unramified at $\mathfrak{p}$, then it vanishes on the inertia $I(\underset{K\subseteq K'}{\varprojlim}\mathfrak{P}')$. Since $\bar{\mathbb{Q}}_{\ell}^{\times}$ is abelian, the value of $\rho$ is constant on the conjugacy class. Combining this two facts we get that in this case $\rho(Fr_{\mathfrak{p}})$ has well defined value.
\end{rem}

We are finally ready to state the unramified version of global class field theory over function fields:

\begin{thm}[Unramified class field theory]\label{E:UCFT}
In the above notations:
\begin{enumerate}
\item
For each character $\xi:K^{\times}\backslash\mathbb{A}_K^{\times}/\mathbb{O}_K^{\times} \rightarrow\bar{\mathbb{Q}}_{\ell}^{\times}$ there exists a
unique continuous unramified character $\rho:G_K\rightarrow\bar{\mathbb{Q}}_{\ell}^{\times}$ such that  $\rho(Fr_v)=\xi(\pi_v)$ for all $v$.

\item
For each continuous unramified character $\rho:G_K\rightarrow\bar{\mathbb{Q}}_{\ell}^{\times}$ there exists a unique character
$\xi:K^{\times}\backslash\mathbb{A}_K^{\times}/\mathbb{O}_K^{\times} \rightarrow\bar{\mathbb{Q}}_{\ell}^{\times}$
such that  $\rho(Fr_v)=\xi(\pi_v)$ for all $v$.
\end{enumerate}
\end{thm}

To deal with ramified extensions, we must first define a notion that quantify the ramification of an extension at some prime:

\begin{defn} [Higher ramification groups]
\begin{enumerate}
See also \cite{Se3} for this definition and properties.

\item
Let $L$ be complete with respect to a nontrivial discrete valuation. $L'$ a finite Galois extension of $L$ with Galois group $G_{L'/L}$. Then the
{\em u lower numbered ramification group}, $u\geq-1$ of this extension is $G_{L'/L,u}=\{\sigma\in G_{L'/L}:\forall a\in \mathcal{O}_{L'},
\space v_L(\sigma(a)- a)\geq u+1 \}$.

\item
In the same situation define the Herbrand function $\phi_{L'/L}:[-1,\infty)\rightarrow[-1,\infty)$ by:
\begin{center}
$   \phi_{L'/L}(u)= \left\{
     \begin{array}{lr}
       \int_0^u\frac{dt}{[G_{L'/L,0}:G_{L'/L,t}]} &  0\leq u\\
       u &  -1\leq u\leq 0
     \end{array}
   \right.$
\end{center}
this function is clearly increasing (and piecewise linear), therefore it has an inverse.

\item
Define the {\em v upper numbered ramification group} by $G^v_{L'/L}:=G_{L'/L,u}$ where $v=\phi_{L'/L}(u)$.

\item
If $L''\supseteq L'\supseteq L$ are finite Galois extension of complete nontrivial DVR. Then the quotient map
$G_{L''/L}\rightarrow G_{L'/L}$ induces maps $G^i_{L''/L}\rightarrow G^i_{L'/L}$ (see \cite{Se3}). So we define the {\em higher ramification
groups} of $L$ by $G_L^i=\underset{L\subseteq L'}{\varprojlim} G^i_{L'/L}$.
\end{enumerate}
\end{defn}

\begin{rem}\label{E:inertia}
Note that $G_{L'/L}^0$ is precisely the inertia group $I(\mathfrak{P})$ defined above, where $L'$ is the completion of $K'$ at the prime $\mathfrak{P}$.
\end{rem}

\begin{defn}
\begin{enumerate}
\item
A {\em $p$ group} is a finite group of order a power of $p$.
\item
A {\em pro-$p$ group} is an inverse limit of $p$-groups.

\end{enumerate}
\end{defn}

\begin{prop}\label{E:PHRG}
With the notations above, the group $G^0_{L'/L}/G^1_{L'/L}$ is isomorphic to a subgroup of $\mathbb{F}_q^{\times}$ hence it is a finite group of order prime to $p$.
$G^i_{L'/L}$ for $i\geq1$ are $p$-groups.
\end{prop}

\begin{proof}
\cite[Chapter IV, $\S2$, Corollaries 1,3]{Se3}.
\end{proof}

With this definition in hand, we can finally measure the ramification of field extensions:

\begin{defn}[Ramification]
Let $\rho:G_K \rightarrow \bar{\mathbb{Q}}_{\ell}^{\times}$ be a character of $G_K$ in the field of $\ell$-adic numbers (or any other field) and let $\mathfrak{p}$ be a prime of $K$. Recall that the decomposition group $D(\mathfrak{p})$ is defined up to conjugacy in $G_K$ and it is isomorphic to $G_{K_{\mathfrak{p}}}$ (it will sometimes be denoted by $G_{\mathfrak{p}}$).

\begin{enumerate}
\item
We say that $ram_{\mathfrak{p}}(\rho)\leq i$ if $\rho$ vanishes on the higher ramification group $G^i_{\mathfrak{p}}$.

\item
Assume $[D]=\sum n_i\mathfrak{p}_i$ is an effective divisor on $K$ (i.e. a finite formal sum of its primes with $n_i\geq0$), we say that
$ram(\rho)\leq [D]$ if for all $i$ $ram_{\mathfrak{p}_i}(\rho)\leq n_i$.

\item
We say that $\rho$ is {\em tamely ramified}, if its ramification at any prime $\mathfrak{p}$ bounded by 1. Otherwise, it is {\em wildly ramified}.

\end{enumerate}
\end{defn}

\begin{rem}
We already defined unramified characters, note that $\rho$ is unramified iff its ramification is bounded by 0 at any prime $\mathfrak{p}$, this is true because of remark \re{inertia}.
\end{rem}

We need the following notation in order to state the ramified version of theorem \re{UCFT}:

\begin{defn}
Let $[D]=\sum n_iv_i$ be an effective divisor on $K$.

\begin{enumerate}
\item
The {\em support} of $[D]$ is the set of valuations that occur with nonzero coefficient in $[D]$, it is denoted $Supp(D)$.

\item
Define $U_{n_i}=\{x\in\mathcal{O}_{v_i}^{\times}:x\equiv1 \mod \mathfrak{m}_{v_i}^{n_i}\}$.

\item
$\mathbb{O}_D^{\times}$ will denote the product $\underset{v_i\in Supp(D)}{\prod}U_{n_i}\times\underset{v\notin Supp(D)}{\prod}\mathcal{O}_v^{\times}$.
\end{enumerate}
\end{defn}

Now we are ready to generalize the theorem \re{UCFT} to the ramified case:

\begin{thm}[Ramified class field theory]\label{E:RCFT}
In the above notations:
\begin{enumerate}
\item
For each character $\xi:K^{\times}\backslash\mathbb{A}_K^{\times}/\mathbb{O}_D^{\times} \rightarrow\bar{\mathbb{Q}}_{\ell}^{\times}$ there exists a
unique continuous character $\rho:G_K\rightarrow\bar{\mathbb{Q}}_{\ell}^{\times}$ with $ram(\rho)\leq [D]$ and $\rho(Fr_v)=\xi(\pi_v)$ for all primes $v\notin Supp(D)$.

\item
For each continuous character $\rho:G_K\rightarrow\bar{\mathbb{Q}}_{\ell}^{\times}$ with $ram(\rho)\leq [D]$ there exists a unique character
$\xi:K^{\times}\backslash\mathbb{A}_K^{\times}/\mathbb{O}_D^{\times} \rightarrow\bar{\mathbb{Q}}_{\ell}^{\times}$
such that  $\rho(Fr_v)=\xi(\pi_v)$ for all primes $v\notin Supp(D)$.
\end{enumerate}
\end{thm}

The following theorem will be useful immediately:

\begin{thm}[Chebotarev's density theorem]
Let $K'/K$ be a finite extension, $C$ a conjugacy class of $Gal(K'/K)$. The set of primes $v$ unramified in $K'$ such that the conjugacy class of $Fr_v$ is in $C$ has density $\#C/\#Gal(K'/K)$.
\end{thm}

\begin{proof}
\cite[Theorem 6.3.1]{FJ}.
\end{proof}

\begin{proof}[Proof of uniqueness in \re{RCFT}]
\begin{enumerate}
\item
Since $\bar{\mathbb{Q}}_{\ell}^{\times}$ is abelian, it is enough to check the value of $\rho$ on conjugacy classes. By continuity of $\rho$, in order to prove that it is unique it is enough to show that it is uniquely defined on some dense subset of $G_K$. We know by the theorem how it is defined on the conjugacy classes of $Fr_v$ $v\notin Supp(D)$ (namely $\rho(Fr_v)=\xi(\pi_v)$). Hence we must prove that the set of such conjugacy classes of $Fr_v$ is dense. By definition of the Krull topology, it is enough to prove that if $K'/K$ is a finite extension, then every element of $Gal(K'/K)$ is the image of some $Fr_v$ (i.e. image of any element of this conjugacy class) under the quotient map $G_K\rightarrow Gal(K'/K)$. By functoriality of the Frobenius, it is enough to show that any element of $Gal(K'/K)$ is in a conjugacy class of some Frobenius, but Chebotarev's density theorem states that this is indeed the case (actually, for infinitely many places $v$).

\item
Let $x\in\mathbb{A}_K^{\times}$, we must show that it is generated by the elements $\pi_v$ $v\notin Supp(D)$. By the weak approximation theorem we can find an element $y\in K^*\subseteq \mathbb{A}_K^{\times}$ such that $y\equiv x \mod \mathfrak{m}_{v_i}^{n_i}$ for all $v_i\in Supp(D)$, hence $(x/y)_{v_i}\in U_{n_i}$ but $x/y$ is equivalent to $x$ as an element of $K^{\times}\backslash\mathbb{A}_K^{\times}/\mathbb{O}_K^{\times}$, therefore we can replace $x$ by $x/y$ and assume that $(x)_{v_i}\in U_{n_i}$ for all $v_i\in Supp(D)$.

If we denote $val_v(x)$ the value of $x$ under the valuation defined by $v$ we have that as elements of $K^{\times}\backslash\mathbb{A}_K^{\times}/\mathbb{O}_K^{\times}$, $x$ is equivalent to $\underset{v\notin Supp(D)}{\prod}\pi_v^{val_v(x)}$, note that this product is actually finite by definition of $\mathbb{A}_K^{\times}$. We conclude that $\pi_v$ $v\notin Supp(D)$ generates $K^{\times}\backslash\mathbb{A}_K^{\times}/\mathbb{O}_K^{\times}$.

$\xi$ is a homomorphism, hence it is determined by its values on the generators $\pi_v$ $v\notin Supp(D)$. Since the theorem requires that $\xi(\pi_v)=\rho(Fr_v)$ it defines $\xi$ uniquely on $K^{\times}\backslash\mathbb{A}_K^{\times}/\mathbb{O}_K^{\times}$.
\end{enumerate}
\end{proof}

Therefore it remains to prove only the existence.

\section{The generalized Picard group}\label{S:GPG}

\begin{rem}
In this section we refer to the Zariski topology and not to the \'etale topology, unless otherwise stated explicitly.
\end{rem}

Our next goal is formulating the results of class field theory in terms of algebraic curves, so that the machinery of algebraic geometry will help us prove them. The basic connection between field extensions, which we want to understand, and algebraic geometry is provided by the following result:

\begin{thm}\label{E:FFAC}
Let $k$ be an algebraically closed field, there exists an equivalence of categories between the following:
\begin{enumerate}
\item (Function fields of transcendence degree one over $k$ and $k$ morphisms)$^{op}$.

\item Nonsingular connected projective curves and dominant morphisms.

\item Connected quasi-projective curves and dominant rational maps.

\end{enumerate}
\end{thm}

\begin{proof}
\cite[Chapter I, Corollary 6.12]{Ha}.
\end{proof}

The symbol $^{op}$ in the above theorem, means the opposite category: the same category with arrows in the reverse direction.

We remind this correspondence here in a few words, a more detailed account can be found e.g. in \cite{Ha}. The connection between (2) and (3) is that every quasi-projective curve is birationally isomorphic to a unique nonsingular projective model.

Also (2) corresponds to (1) by passing to the field of rational functions. If $f:X\rightarrow Y$ is a dominant morphism of nonsingular curves, then there exists a morphism  between the fields of rational functions $K(f):K(Y)\rightarrow K(X)$ defined by $\phi\mapsto\phi\circ f$.

The following correspond to each other under the correspondence of the above theorem To get the classical result we will use a similar correspondence for corves over the finite field $\mathbb{F}_q$, The following correspond to each other under the correspondence of the above theorem:
\begin{center}
    \begin{tabular}{ | l | l |}
    \hline
    Fields of transcendence degree 1 over $k$  & Nonsingular projective curves over $k$ \\ \hline
    Morphism of $k$ algebras   & Dominant morphism \\ \hline
    Unramified field extension   & \'Etale morphism \\ \hline
    Finite Galois extension  & Finite Galois covering \\ \hline
    Primes of  $K$   & Points of $X$ \\ \hline

    \hline
    \end{tabular}
\end{center}

From now on unless otherwise stated explicitly, let $X$ be a nonsingular projective connected curve over the the separably closed field $k$.

The next thing to do is to find the geometric analog of $K^{\times}\backslash\mathbb{A}_K^{\times}/\mathbb{O}_D^{\times}$.

From now on, let $D$ be a finite subscheme of $X$ and $[D]=\underset{i}{\sum} n_ix_i$ the corresponding
finite formal sum where $x_i\in|X|$ are different closed points and $n_i\in\mathbb{N}$ for all $i$. Recall that closed points of $X$ correspond
to places of $K$ (the function field of $X$), and elements of $K$ correspond to rational functions. Note also that
for any prime $v$, $K_v^{\times}/\mathcal{O}_v^{\times}\cong\mathbb{Z}$,
hence the group $\mathbb{A}_K^{\times}/\mathbb{O}_K^{\times}$ is the free abelian group
generated by the primes of $K$. This motivates the following definition.

\begin{defn}[Generalized Weil divisors]
\begin{enumerate}
\item
A {\em Weil divisor on $X-D$} is defined as an element of $Div(X-D)$ - the free
abelian group generated by the closed points of $X-D$.

\item
Let $f\in \mathcal{K}_D^{\times}(X)$ i.e. a nonzero rational function
such that $\forall i:f\equiv1\mod \mathfrak{m}_{x_i}^{n_i}$ (we also denote
it by $f\equiv1\mod D$).
We can associate to $f$ a Weil divisor on $X-D$:
$(f)=\sum_{x\in X}ord_x(f)x$, where $ord_x(f)$ is the order of the zero
of $f$ in $x$. Weil divisors of the form $(f)$ are called {\em $D$-principal}.
The $D$-principal divisors form a subgroup of $Div(X-D)$ since if
$f_1,f_2\in \mathcal{K}_D(X)^{\times}$ then $(f_1/f_2)=(f_1)-(f_2)$. Two divisors are called {\em $D$-linearly equivalent} if their difference is $D$-principal.

\item
The quotient group of $Div(X-D)$ by the $D$-principal divisors is called the {\em $D$-Weil class group}
of $X$ and is denoted by $Cl_D(X)$.

\item
Define the {\em degree map} $deg:Div(X-D)\rightarrow \mathbb{Z}$ by
$deg(\sum b_ix_i)=\sum b_i$ recall that for any  rational function $f$,
$deg((f))=\sum_{x\in X}ord_x(f)=0$, hence there exists a well defined map on the quotient
$deg:Cl_D(X)\rightarrow \mathbb{Z}$. The inverse image of $d\in \mathbb{Z}$ under this map
if the set of degree $d$ divisors (under the equivalence $D_1\sim D_2$ iff
$D_1-D_2=(f)$ for some $f\in \mathcal{K}^{\times}_D(X)$), and is denoted also by $Cl^d(X)$.

\item
A Weil divisor $D'=\sum b_ix_i$ is called {\em effective},
if $\forall i:b_i\geq0$.

\item
Let $D_1,D_2$ be two Weil divisors, we say that $D_1\geq D_2$ iff $D_1-D_2$ is effective. We will write $f_1\equiv f_2 \mod D'$ if $(f_1-f_2)\geq D'$ as elements of $Div(X)$.
\end{enumerate}
\end{defn}

\begin{rem}
Points of $X$ correspond to primes of $\mathcal{O}_K$ and elements of $K$ are rational functions on $X$. We obtain that in the case when $X$ defined over $\mathbb{F}_q$, $Cl_D(X)$ is exactly $K^{\times}\backslash\mathbb{A}_K^{\times}/\mathbb{O}_D^{\times}$ defined above.
\end{rem}

From now on $\mathcal{O}$ denotes the sheaf of regular functions on $X$.

Another useful way, which is more geometric, to look at divisors
is using line bundles. We define it here:

\begin{defn} [Generalized Picard group]
The {\em generalized Picard group} $Pic_D(X)$ of $X$
Is the moduli space of pairs $(\mathcal{L},\psi)$ where $\mathcal{L}$ is a line bundle on $X$, and $\psi:\mathcal{L}|_D\xrightarrow{\approx} \mathcal{O}|_D$ is a trivialization over $D$ (which mean a map $\psi=\prod \psi_i$ where $\psi_i:\mathcal{L}_{x_i}/\mathfrak{m}_{\mathcal{L},x_i}^{n_i} \xrightarrow{\approx}\mathcal{O}_{x_i}/\mathfrak{m}_{x_i}^{n_i}$, $\mathfrak{m}_{\mathcal{L},x_i}$ is the maximal ideal of $\mathcal{L}_{x_i}$). Two such pairs $(\mathcal{L},\psi)$ and $(\mathcal{L}',\psi')$ are equivalent if there exist an isomorphism $\theta:\mathcal{L}\xrightarrow{\approx}\mathcal{L}'$ such that the following diagram commutes:

\begin{center}
$\xymatrix{
\mathcal{L}|_D \ar[rr]^{\theta|_D} \ar[rd]_{\psi} &&
\mathcal{L}'|_D  \ar[ld]^{\psi'}\\
&\mathcal{O}|_D}$
\end{center}

The group operation of the generalized Picard group is $\otimes$ of line bundles, $(\mathcal{L},\psi)\otimes (\mathcal{L}',\psi')=(\mathcal{L}'',\psi'')$ where $\mathcal{L}''=\mathcal{L}\otimes\mathcal{L}'$ and $\psi''$ is the composition $\mathcal{L}|_D\otimes\mathcal{L}'|_D \xrightarrow{\psi\otimes\psi'}\mathcal{O}|_D\otimes\mathcal{O}|_D \xrightarrow{a\otimes b\mapsto ab}\mathcal{O}|_D$.
It makes the $Pic_D(X)$ into an abelian group.
\end{defn}

\begin{rem}
\begin{enumerate}
\item
The Picard group can be given a structure of a commutative algebraic group, it is done using the functorial approach to moduli schemes. The proof in the case $D=\emptyset$ can be found in \cite[Theorem 4.8]{Kl}.

Later, in propositions \re{UFIBER}, \re{RFIBER} and \re{FPIC} and also in the proof of \re{CHAR0}, we will show that fibers of specific maps involving Picard schemes are isomorphic to an affine space, projective space ar other variety as sets, all these isomorphisms  are also isomorphisms of schemes, to prove it the functorial approach must be used.

\item
In the above definitions, if we take $D=\emptyset$ we obtain the usual well known
examples, the class group and
Picard group of $X$, they will be denoted simply by $Cl(X)$ and $Pic(X)$.

\item
We defined the Picard group in terms of line bundles. Any vector bundle defines
a locally constant sheaf by taking its sheaf of sections, this operation is an
equivalence of categories between the category of $k$-vector bundles
of rank $r$ and the category of
locally constant sheaves of $k$-vector spaces with dimension $r$
(proof of this claim can be found e.g. in \cite[Proposition 1.8.1]{LP}). Thus, in what follows we freely switch between these two categories
and sometimes think on the Picard group in terms of locally constant sheaves of
rank 1, i.e. invertible sheaves.
\end{enumerate}
\end{rem}

\begin{note}
When the curve $X$ will be clear from the context, we sometimes omit it from the notation and write $Cl_D$ and $Pic_D$
instead of $Cl_D(X)$ and $Pic_D(X)$, respectively.
\end{note}

If $X$ is defined over $\mathbb{F}_q$, which is the most interesting case as far as class field theory concerned, The Picard group is just the automorphic side of class field theory, $K^{\times}\backslash\mathbb{A}_K^{\times}/\mathbb{O}_D^{\times}$:

\begin{prop}\label{E:GPGI}
There exists an isomorphism of groups between
$K^{\times}\backslash\mathbb{A}_K^{\times}/\mathbb{O}_D^{\times}$ and the set of $\mathbb{F}_q$-points of $Pic_D$, $Pic_D(\mathbb{F}_q)$.
\end{prop}

\begin{proof}
Consider first the group $\mathbb{B}$ of triples of the form $(\mathcal{L},\{(\phi_x)\}_{x\in|X|},\phi_K)$ where $\mathcal{L}$ is a line bundle on $X$. Since $\mathcal{L}$ is locally constant, for any closed point $x\in|X|$ we can trivilize it over $Spec\,\mathcal{O}_x$, let $\phi_x:\mathcal{L}|_{Spec\,\mathcal{O}_x}\xrightarrow{\approx}\mathcal{O}_x$ be such a trivilization. We can also trivialize it over some small enough open neighborhood,
and therefore over the generic point $Spec\,K$, let us denote such a trivialization by
$\phi_K:\mathcal{L}|_{Spec\,K}\xrightarrow{\approx}K$.

$\mathbb{B}$ becomes a group under the operation of tensor product, if  $(\mathcal{L}_i,\{(\phi_{ix})\}_{x\in|X|},\phi_{iK}) \in \mathbb{B}$ for $i=1,2$ then their product $(\mathcal{L},\{(\phi_{x})\}_{x\in|X|},\phi_{K})$ is defined to be $\mathcal{L}=\mathcal{L}_1\otimes \mathcal{L}_2$,
$\phi_x$ is the composition:
\begin{center}
$\mathcal{L}_1|_{Spec\,\mathcal{O}_x}\otimes\mathcal{L}_2|_{Spec\,\mathcal{O}_x}\xrightarrow{\phi_{1x}\otimes\phi_{2x}} \mathcal{O}_x\otimes\mathcal{O}_x\xrightarrow{s\otimes s'\mapsto ss'}\mathcal{O}_x$
\end{center}
and similarly for $\phi_K$.

We claim first that there exist an isomorphism of groups $\alpha:\mathbb{B}\rightarrow \mathbb{A}^{\times}_K$ defined as follows: let $(\mathcal{L},\{(\phi_{x})\}_{x\in|X|},\phi_{K}) \in\mathbb{B}$ restricting each $\phi$ to $Spec\,K_x$ we get isomorphisms
$g_x=\phi_x\circ\phi_K^{-1}:Spec\,K_x\xrightarrow{\approx}Spec\,K_x$, these are elements of $K^{\times}_x$ for all $x$, since $\mathcal{L}$ is locally constant, almost all $g_x$ are in $\mathcal{O}^{\star}_x$ therefore this is an element of $\mathbb{A}_K^{\times}$.

$\alpha$ is homomorphism: let $(\mathcal{L}_i,\{(\phi_{ix})\}_{x\in|X|},\phi_{iK}) \in \mathbb{B}$ for $i=1,2$ then consider the following isomorphisms:
\begin{center}
$K_x\xrightarrow{ss'\leftmapsto s\otimes s'}K_x\otimes K_x\xrightarrow{\phi^{-1}_{1K}\otimes\phi^{-1}_{2K}} \mathcal{L}_1|_{Spec\,K_x}\otimes\mathcal{L}_2|_{Spec\,K_x}\xrightarrow{\phi_{1x}\otimes\phi_{2x}} K_x\otimes K_x\xrightarrow{s\otimes s'\mapsto ss'}K_x$
\end{center}
and under these maps we have:
\begin{center}
$1\mapsto 1\otimes 1 \mapsto \phi_{1K}^{-1}(1)\otimes {\phi}_{2K}^{-1}(1)\mapsto g_{1x}(1)\otimes g_{2x}(1)\mapsto g_{1x}(1)g_{2x}(1)=g_{1x}g_{2x}$
\end{center}
if $(\mathcal{L},\{(\phi_{x})\}_{x\in|X|},\phi_{K}) =(\mathcal{L}_1,\{(\phi_{1x})\}_{x\in|X|},\phi_{1K})\cdot (\mathcal{L}_2,\{(\phi_{2x})\}_{x\in|X|},\phi_{2K})$ the above map shows that $g_x=\phi_x\circ\phi_K^{-1}=g_{1x}g_{2x}$ hence $\alpha$ is an homomorphism.

$\alpha$ is injective. Indeed, assume $(\mathcal{L},\{(\phi_{x})\}_{x\in|X|},\phi_{K})\in \ker(\alpha)$, therefore $g_x=\phi_x\circ \phi_K^{-1}|_{Spec\,K_x}=1$ the identity, hence $\phi_x,\phi_K:\mathcal{L}|_{Spec\,K_x}\rightarrow K_x$ are compatible, and by the gluing property we can glue them to get $\phi:\mathcal{L}\xrightarrow{\approx}\mathcal{O}$, therefore $(\mathcal{L},\{(\phi_{x})\}_{x\in|X|},\phi_{K})$ is trivial.

$\alpha$ is also surjective, let $\{(g_{x})\}_{x\in|X|} \in \mathbb{A}^{\times}_K$ since by gluing to define a sheaf $\mathcal{L}$ it is sufficient to define $\mathcal{L}_{\mathcal{O}_x}=\mathcal{L}|_{Spec\,\mathcal{O}_x}=\mathcal{O}_x$, $\mathcal{L}_{K}=\mathcal{L}|_{Spec\,K}=K$ and give gluing isomorphisms on the intersections: $\mathcal{L}_{\mathcal{O}_x}|_{Spec\,K_x}\xrightarrow{\approx} \mathcal{L}_{K}|_{Spec\,K_x}$. We take this gluing isomorphisms to be the required $g_x$ the sheaf $\mathcal{L}$ we obtain is locally constant since $(g_x)\in \mathbb{A}^{\times}_K$, and $\phi_x,\phi_K$ are defined to be the identifications $\mathcal{L}_{\mathcal{O}_x}=\mathcal{O}_x$ and $\mathcal{L}_{K}=K$, respectively. $\alpha(\mathcal{L},\{(\phi_{x})\}_{x\in|X|},\phi_{K})=\{(g_{x})\}_{x\in|X|}$.

If we consider pairs of the
form $(\mathcal{L},\{(\phi_x)\}_{x\in|X|})$ as quotients of $\mathbb{B}$, under the imorphism $\alpha$ it is
it is the same as introducing the equivalence relation defined by $K^{\times}$ on $\mathbb{A}^{\times}_K$ hence there exists an isomorphism between $K^{\times}\backslash\mathbb{A}_K^{\times}$ and such pairs.

$Pic_D(\mathbb{F}_q)$ is a quotient of such pairs, the quotient map is $(\mathcal{L},\{(\phi_x)\}_{x\in|X|})\mapsto (\mathcal{L},\psi)$ where $\psi|_{Spec\,\mathcal{O}_x}\equiv\phi_x \mod \mathfrak{m}_{x_i}^{n_i}$ here as usual
$\mathfrak{m}_{x_i}$ are the maximal ideals of $\mathcal{O}_{x_i}$. In particular, two such pairs $(\mathcal{L},\{(\phi_x)\}_{x\in|X|})$ and $(\mathcal{L}',\{\phi'_x\}_{x\in|X|})$ maps to the same element of $Pic_D(\mathbb{F}_q)$ if $\mathcal{L}\cong \mathcal{L}'$ and the maps induce the same trivialization $\psi$ over $D$, that is $(\phi_x\circ(\phi'_x)^{-1})\in \mathbb{O}^{\times}_D$. We conclude that $K^{\times}\backslash\mathbb{A}_K^{\times}/\mathbb{O}_D^{\times}\cong Pic_D(\mathbb{F}_q)$ and we are done.
\end{proof}

\begin{note}

We can define a map $\mathcal{O}_D:Div(X-D)\rightarrow Pic_D$ by $x\mapsto (\mathcal{O}(x),\psi)$ where $\mathcal{O}(x)$ is the sheaf of regular function with $(f)\geq [x]$ and $\psi:\mathcal{O}(x)|_D\rightarrow \mathcal{O}_D$ is the identity map, this map induces a canonical isomorphism of group $Cl_D\xrightarrow{\approx} Pic_D$, which will also be denoted by $\mathcal{O}_D$.

\end{note}

\begin{defn}\label{E:PHIDEF}\
\begin{enumerate}
\item
We can define an action $act:(X-D)\times Pic_D\rightarrow Pic_D$ by $act(x, \mathcal{L}):=\mathcal{O}_D([x])\otimes \mathcal{L}$.

\item
We sometimes use the notation $act_x:Pic_D\rightarrow Pic_D$ for $act_x(\mathcal{L})=act(x,\mathcal{L})$.

\item
In particular, taking $(\mathcal{L},\psi)=(\mathcal{O},id)$ we obtain a map $\phi:X-D\rightarrow Pic_D$, defined by $\phi(x)=\mathcal{O}_D([x])$. From now on, the letter $\phi$ will be used only for this map.
The image of $\phi$ is actually contained in the connected component $Pic_D^1$ of $Pic_D$ defined by line bundles of degree 1.
\end{enumerate}
\end{defn}

The following generalization of the classical notion will be very important for us later:

\begin{defn}
A {\em complete linear system} is the set of effective divisors on $X-D$ $D$-equivalent to some $D_0$. It is denoted by $|D_0|_D$.
\end{defn}

\begin{note}
$l(D_0)$ will denote $dim_k(H^0(X,\mathcal{O}(D_0)))$.
\end{note}

\begin{prop}\label{E:UFIBER}
In the case $[D]=0$, $D_0$ is a Weil divisor. The set $|D_0|$ is in one to one correspondence with the projective space $H^0(X,\mathcal{O}(D_0))/k^{\times}\cong \mathbb{P}^{l(D_0)-1}$.
\end{prop}

\begin{proof}
\cite[Chapter II, Proposition 6.13]{Ha}.
\end{proof}

We now want to generalize the above proposition, and calculate the complete linear system in the ramified case, $[D]>0$.

\begin{prop}\label{E:RFIBER}
Assume $Supp(D_0)\bigcap Supp(D)=\emptyset$, $[D]>0$ and $D_0\geq 0$ then the elements of the complete linear system $|D_0|_D$ corresponds to the points of the affine space $H^0(X,\mathcal{O}(D_0-[D]))\cong\mathbb{A}^{l(D_0-[D])}$.
\end{prop}

\begin{proof}
$H^0(X,\mathcal{O}(D_0-[D]))$ is the set of global sections of the sheaf $\mathcal{O}(D_0-[D])$, such global sections are just rational functions $f\in K$ that satisfies for any point $x\in |X|$ $ord_x(f)\geq ord_x([D]-D_0)$, this is simply the statement that $(f)\geq [D]-D_0$ or $(f)+D_0-[D]$ is effective.

We know that the constant function $1\in K$ is in $H^0(X,\mathcal{O}(D_0))$ since $D_0$ is effective. Hence we can define a map $H^0(X,\mathcal{O}(D_0-[D]))\rightarrow H^0(X,\mathcal{O}(D_0))$ by $f\mapsto f+1$.

Also note that since $Supp(D_0)\bigcap Supp(D)=\emptyset$, a function $f\in H^0(X,\mathcal{O}(D_0-[D]))$ satisfies $f\equiv 0 \mod [D]$, hence $f+1\equiv1 \mod [D]$.

On the other hand, $f+1\in H^0(X,\mathcal{O}(D_0))$ therefore $(f+1)+D_0$ is effective. We obtain a map $h:H^0(X,\mathcal{O}(D_0-[D]))\rightarrow |D_0|_D$ defined by $h(f)=(f+1)+D_0$. If we prove that $h$ is bijective we are done.

$h$ is injective: we claim that if for two functions $f_1,f_2\in H^0(X,\mathcal{O}(D_0-[D]))$ satisfies $(f_1+1)=(f_2+1)$ then $f_1=f_2$. Indeed, $(\frac{f_1+1}{f_2+1})=0$ the zero divisor, it implies that $\frac{f+1}{f'+1}\in H^0(X, \mathcal{O}^{\times})$ is a (nonvanishing) regular function, but $X$ is projective, so a regular function is necessarily constant. Hence $\frac{f_1+1}{f_2+1}=\lambda\in k^{\times}$ and it follows that the difference $f_1-\lambda f_2=\lambda-1$ is a constant function. But if $[D]>0$ and $Supp(D_0)\bigcap Supp(D)=\emptyset$ then the only constant function in $H^0(X,\mathcal{O}(D_0-[D]))$ is zero, therefore $\lambda=1$ and $f_1=f_2$.

$h$ is surjective: we must show that any function that satisfies $g\equiv 1 \mod [D]$ and $(g)+D_0\geq 0$ is of the form $g=f+1$ where $f\in H^0(X,\mathcal{O}(D_0-[D]))$; this is easy: $(g)+D_0\geq 0$ implies $(f)+D_0\geq 0$ (because $D_0$ is effective and $f,g$ have the same poles). Also, $g\equiv 1 \mod [D]$ implies $f\equiv 0 \mod [D]$ and therefore $ord_x(f)\geq ord_x([D])$ for any $x\in Supp(D)$. We conclude that $(f)-[D]+D_0$ is effective, in other words $f\in H^0(X,\mathcal{O}(D_0-[D]))$, so we are done.
\end{proof}

We now give a relation between the generalized Picard in the ramified and unramified case, but first we define:

\begin{defn}
Define a functor $\mathfrak{R}_D:Sch^{op}\rightarrow Ab$ by:

$\mathfrak{R}_D(T)=$ the group of trivializations of $\mathcal{O}_{T\times X}|_{T\times D}$.

We state without giving a proof that this functor is representable by an algebraic group $R_D$ whose set of $k$-points is isomorphic to $\underset{x_i\in Supp(D)}{\prod}\mathcal{O}_{x_i}^{\times}/(1+\mathfrak{m}_{x_i}^{n_i})$.
\end{defn}

\begin{prop}\label{E:FPIC}
There exist a natural surjective homomorphism of algebraic groups  $Pic_D\rightarrow Pic$ with kernel isomorphic to $R_D/\mathbb{G}_m$.
\end{prop}

\begin{proof}
On the level of points, the kernel of the map $Pic_D\rightarrow Pic$ defined by $(\mathcal{L},\psi)\mapsto \mathcal{L}$ consists of all trivializations $\psi:\mathcal{L}|_D\xrightarrow{\approx} \mathcal{O}|_D$. Such trivialization corresponds to trivializations of the form $\theta:\mathcal{O}|_D\xrightarrow{\approx} \mathcal{O}|_D$, indeed, since $\mathcal{L}|_D$ is trivial we can choose a trivialization $\psi_0:\mathcal{L}|_D\xrightarrow{\approx} \mathcal{O}|_D$ and the map $\theta\mapsto \theta\circ\psi_0$ is bijective because its inverse is composition with $\psi_0^{-1}$. Therefore the fiber is isomorphic to the set of all trivializations $\theta:\mathcal{O}|_D\xrightarrow{\approx} \mathcal{O}|_D$, a trivialization is an automorphism of  $\underset{x_i\in Supp(D)}{\prod} \mathcal{O}_{x_i}/\mathfrak{m}_{x_i}^{n_i}$ preserving each factor, but this is simply multiplication by an element of $R_D=\underset{x_i\in Supp(D)}{\prod}\mathcal{O}_{x_i}^{\times}/(1+\mathfrak{m}_{x_i}^{n_i})$. There is another equivalence on the fiber since by the definition of $Pic_D$ two trivializations of $\mathcal{O}$ are considered the same if they differ by an automorphism of $\mathcal{O}$ i.e. nonzero constant function, hence the fiber is $R_D/\mathbb{G}_m$.
\end{proof}

\begin{rem}
Note that the statement is about algebraic schemes, but we proved a statement about the set of points. Therefore the proof is incomplete, to prove it as algebraic varieties, the functorial approach must be used. The method for the proof will be as follows:
\begin{enumerate}
\item
Define functors $\mathfrak{Pic}_D,\mathfrak{Pic}:Sch^{op}\rightarrow Ab$ by:

$\mathfrak{Pic}_D(T)=$ the group of invertible sheaves $\mathcal{L}$ on $T\times X$ with trivialization over $T\times D$, divided by the group of invertible sheaves on $T\times X$ which are pullbacks of a sheaf on $T$ under the projection map $T\times X\rightarrow T$.

$\mathfrak{Pic}(T)=$ the group of invertible sheaves $\mathcal{L}$ on $T\times X$.

\item
Show that these functors are representable, $Pic$ and $Pic_D$ are by definition the schemes representing $\mathfrak{Pic}$ and $\mathfrak{Pic}_D$ respectively.

\item
Give a natural transformation between the functors $\mathfrak{Pic}_D\rightarrow \mathfrak{Pic}$, it is simply the transformation that forgets the trivialization.

\item
To conclude, show that natural transformation induces a surjective homomorphism between the representing scheme $Pic_D\rightarrow Pic$, with kernel $R_D$.
\end{enumerate}
\end{rem}

\section{The \'etale fundamental group}\label{S:EFG}

On our way to find geometric analogs of notions from class field theory, we now want to define the analog of the Galois group. The group $Gal(K^{un}/K)$, the maximal unramified extension, turns out to be a more natural geometric object than $Gal(\bar{K}/K)$, so we define the geometric analog of $Gal(K^{un}/K)$. Recall that finite Galois unramified extensions of $K$ corresponds to finite \'etale Galois covering of $X$, since the Galois group is the group of automorphisms of the covering, the following definition is very natural:

\begin{defn}[The \'etale fundamental group]
Let $X$ be a connected scheme, $\bar{x}$ a geometric point. Define a functor from the category of finite \'etale coverings of $X$ to the category of sets $F:FEt/X\rightarrow sets$ by $F(Y)=Hom_X(\bar{x},Y)$. $Aut_X(Y)$ acts on $F(Y)$ such that if $Y$ is connected, then the action is faithful. If the action is also transitive, then $Y$ is called Galois over $X$. The functor $F$ is pro-representable, i.e there exist a projective system $\{X_j\rightarrow X_i:i<j\in I\}$ indexed by $I$ of finite \'etale coverings $p_i:X_i\rightarrow X$ such that $F(Y)=\underset{i}{\varinjlim}Hom_X(X_i,Y)$.  The projective system $\widetilde{X}=(X_i)_{i\in I}$ can be chosen such that the coverings $p_i:X_i\rightarrow X$ are Galois coverings of $X$ since the Galois objects are a cofinal system: for any finite \'etale covering $Y'\rightarrow X$, there exist $Y\rightarrow Y'$ finite \'etale, such that $Y\rightarrow X$ is Galois. We define the {\em fundamental group} of $X$ to be:
If $i<j\in I$ then the map $X_j\rightarrow X_i$ induces a map $Aut_X(X_j)\rightarrow Aut_X(X_i)$ therefore this is also a projective system, and we can define:  $\pi_1(X,\bar{x})=Aut_X(\widetilde{X}):=\underset{i}{\varprojlim}Aut_X(X_i)$. There are some statements in this definition given here without proofs, the proofs can be found in \cite{Mu1}.
\end{defn}

This fundamental group is the group of automorphisms of the "maximal" \'etale covering, exactly as the group $Gal(K^{un}/K)$ is the automorphism group of the maximal unramified field extensions of $K$. Therefore, in the unramified case of class field theory - this is exactly the group $G_K$ we want to explore. We shall now give some properties of the fundamental group, which will be useful later.

\begin{rem}
The fundamental group is independent of the base point, $\pi_1(X,\bar{x})\xrightarrow{\approx}\pi_1(X,\bar{x}')$ canonically up to inner automorphism of $\pi_1(X,\bar{x})$.
\end{rem}

\begin{note}
Following the above remark, we will sometimes omit the base point from the notation and write simply $\pi_1(X)$ instead of $\pi_1(X,\bar{x})$.

When we refer to the fundamental group of a disconnected scheme (such as $\pi_1(Pic(X))$), the meaning is the fundamental group of one of its connected components; in all cases considered the components will have isomorphic fundamental groups. Sometimes the reference will be to some specific connected component, and it will be clear which one from the context.
\end{note}

\begin{con} [Functoriality of the fundamental group]

Let $f:Y\rightarrow X$ be a morphism such that $f(\bar{y})=\bar{x}$. we have a contravariant functor $f^*:FEt/X\rightarrow FEt/Y$ defined by base change $f^*(X')=X'\underset{X}{\times}Y$. If $\alpha\in\pi_1(Y,\bar{y})$ we can define an element $f_*(\alpha)\in\pi_1(X,\bar{x})$ as follows: if $X'\rightarrow X$ is a finite Galois covering, then $\alpha$ gives an automorphism $\bar{\alpha}$ of any finite etale  $X'\underset{X}{\times}Y\rightarrow Y$ (it is just its image under the quotient map $\pi_1(Y)\twoheadrightarrow \pi_1(Y)/\pi_1(X'\underset{X}{\times}Y)$). $\bar{\alpha}$ induces an automorphism $f_*(\bar{\alpha})$ of $X'\rightarrow X$ since by property of base change, the fibers of $X'\underset{X}{\times}Y\rightarrow Y$ are naturally isomorphic to the fibers of $X'\rightarrow X$. We obtain a compatible set of automorphisms for all finite Galois coverings $X'\rightarrow X$ (compatibility follows from the fact that all automorphism came from quotients of the same $\alpha$), this is the required element $f_*(\alpha)\in\pi_1(X,\bar{x})$.
\end{con}

We can translate between properties of coverings and the morphism  $f_*:\pi_1(Y,\bar{y})\rightarrow\pi_1(X,\bar{x})$, for example:

\begin{prop}\label{E:Sur}
$f_*:\pi_1(Y,\bar{y})\rightarrow\pi_1(X,\bar{x})$ is surjective iff for any connected \'etale covering $X'\rightarrow X$, $X'\underset{X}{\times}Y$ is connected.
\end{prop}

\begin{proof}
\cite[Chapter V, 5.2.4]{Mu1}.
\end{proof}

We want to understand what happens to the fundamental group under a finite \'etale covering.

\begin{prop}\label{E:FEC}
 Let $\pi:Y\rightarrow X$ be a finite \'etale covering and $\bar{y}\in Y, \bar{x}\in X$ are geometric points such that $\pi(\bar{y})=\bar{x}$. Then $\pi_*:\pi_1(Y,\bar{y})\hookrightarrow \pi_1(X,\bar{x})$ is injective, and it is of finite index.
\end{prop}

\begin{proof}
The fundamental group $\pi_1(Y)$ is by definition the group of automorphisms of the universal covering $p_i:X_i\rightarrow Y$, $\widetilde{X}=(X_i)_{i\in I}$, taking the composition $f\circ p_i:X_i\rightarrow X$ this is also the universal covering of $X$. But clearly any $Y$-automorphism is also $X$-automorphism so $Aut_Y(X_i)\leq Aut_X(X_i)$. These inclusions induces the map on the fundamental groups $\pi_1(Y)=\underset{i}{\varprojlim}Aut_Y(X_i)\hookrightarrow \pi_1(X)=\underset{i}{\varprojlim}Aut_X(X_i)$. It is of finite index since we can consider a finite \'etale covering $Y'\rightarrow Y$ such that $Y'\rightarrow X$ is Galois. In this case $\pi_1(Y')\lhd\pi_1(X)$ is a normal subgroup, and the quotient is just $Aut_X(Y')$ which is a finite group since its size equal the degree of the covering $Y'\rightarrow X$. But $\pi_1(Y')\leq\pi_1(Y)\leq\pi_1(X)$ and the fact that $\pi_1(Y')\leq\pi_1(X)$ is of finite index implies that $\pi_1(Y)\leq\pi_1(X)$ is also of finite index, and we are done.
\end{proof}

\begin{exmpl}
The projective line over a separably closed field $k$ is simply connected ($\pi_1(\mathbb{P}_k^1,\bar{x})$ is trivial).
\end{exmpl}

In order to prove it we will use the following theorem:

\begin{thm}[Hurwitz's theorem]
Let $f:Y\rightarrow X$ be finite and separated morphism of curves (as in \cite{Ha}, curve here is nonsingular projective curve over an algebraically closed field $k$) of degree $n$ then $2g(Y)-2=n(2g(X)-2)+degR$, where $g$ stands for genus and $R$ is the ramification divisor of $f$.
\end{thm}

\begin{proof}
\cite[Chapter IV, Corollary 2.4]{Ha}.
\end{proof}

\begin{proof}[Proof that $\mathbb{P}^1$ is simply connected]
We want to show that any finite \'etale covering of $\mathbb{P}^1$ is trivial. Let $f:Y\rightarrow \mathbb{P}^1$ be a connected finite \'etale covering. Finite morphisms are proper, \'etale morphisms are smooth. Hence $Y$ is smooth, irreducible (since connected and smooth) i.e. a curve. Therefore we can use Hurwitz's theorem for $f:Y\rightarrow X=\mathbb{P}^1$. Since \'etale morphisms are unramified we have $R=0$. According to Hurwitz $2g(Y)-2=n(g(\mathbb{P}^1)-2)=-2n$ and the only solution is $n=1, g(Y)=0$, a trivial covering $\mathbb{P}^1\rightarrow \mathbb{P}^1$.
\end{proof}

\begin{exmpl}\label{E:PSC}
$\mathbb{P}^n$, $n\geq2$ is also simply connected.
\end{exmpl}

We prove it using:

\begin{thm}[The connectedness theorem]
Let $Y$ be an irreducible variety, $f:Y\rightarrow\mathbb{P}^n$ a proper morphism and $L\subseteq\mathbb{P}^n$ a linear subvariety of codimension $<dimY$. Then $f^{-1}(L)$ is connected.
\end{thm}

\begin{proof}
\cite[Theotem II 1.4]{DS}.
\end{proof}

\begin{proof}[Proof that $\mathbb{P}^n$ is simply connected]
Let $\mathbb{P}^1\hookrightarrow\mathbb{P}^n$ be an inclusion of a linear subvariety. And $X'\rightarrow\mathbb{P}^n$ be a finite connected \'etale covering. Then $X'\underset{\mathbb{P}^n}{\times}\mathbb{P}^1\rightarrow\mathbb{P}^1$ is connected according to the connectedness theorem. Proposition \re{Sur} above then implies that $\pi_1(\mathbb{P}^1)\rightarrow\pi_1(\mathbb{P}^n)$ is surjective. But by the previous example, $\pi_1(\mathbb{P}^1)$ is trivial, hence also $\pi_1(\mathbb{P}^n)$.
\end{proof}

In the case $k=\mathbb{C}$, the field of complex numbers, we have the \'etale fundamental group $\pi_1(X)$ and the topological fundamental group with respect to the complex topology $\pi_1^{top}(X(\mathbb{C}))$. A relation between the two groups will be useful in calculations of fundamental groups. We can't expect the two groups to be isomorphic in general, since topological covering spaces can be of infinite degree while algebraic covering cannot. But there is still a strong connection provided by:

\begin{thm}[The Riemann's existence theorem]
Let $X$ be a connected scheme over $\mathbb{C}$ and $Y'\rightarrow X(\mathbb{C})$ a finite unramified covering of complex analytic spaces, then there exist a unique finite \'etale covering $Y\rightarrow X$ such that $Y'=Y(\mathbb{C})$. It gives an equivalence of categories between finite \'etale coverings of $X$, and finite topological covering of $X(\mathbb{C})$.
\end{thm}

\begin{proof}
The case of $X(\mathbb{C})$ complete can be found in \cite{GAGA} and the general case in \cite{GR}.
\end{proof}

\begin{cor}\label{E:RETC}
Let $X$ be a connected scheme over $\mathbb{C}$, then
$\pi_1(X)\cong\hat{\pi}_1^{top}(X(\mathbb{C}))$. Where $\hat{}$ here means profinite completion.
\end{cor}

\begin{proof}
Immediate from the Riemann existence theorem. The finite quotients of $\pi_1^{top}(X(\mathbb{C}))$ are provided by Galois groups of finite extensions, but finite extensions are algebraizable by the existence theorem. And the \'etale fundamental group is exactly the inverse limit of the corresponding groups, i.e. the profinite
completion.
\end{proof}

This result helps us calculate the fundamental groups of curves over
$\mathbb{C}$; the following will give us the result for curves over a field $k$ of
characteristic 0:

\begin{thm}
Let $k\subseteq k'$ be an extension of separably closed fields. $X\rightarrow Spec\,k$ a proper connected scheme. Then the induced map $\pi_1(X\underset{Spec\,k}{\times}Spec\,k')\xrightarrow{\approx} \pi_1(X)$ is an isomorphism.
\end{thm}

\begin{proof}
\cite[Chapter VII, Proposition 7.3.2]{Mu1}.
\end{proof}

\begin{exmpl}\label{E:ASC}
Assume $Char(k)=0$ then the affine space $\mathbb{A}_k^n$ is simply connected.
\end{exmpl}

\begin{proof}
The fact that $\pi_1^{top}(\mathbb{A}_{\mathbb{C}}^n(\mathbb{C}))=\pi_1^{top}(\mathbb{C}^n)={1}$ is well known, by corollary \re{RETC} from the Riemann existence theorem $\pi_1(\mathbb{A}_{\mathbb{C}}^n)=1$. Considering the inclusion $\bar{\mathbb{Q}}\hookrightarrow\mathbb{C}$, we obtain by the above theorem $\pi_1(\mathbb{A}_{\bar{\mathbb{Q}}}^n)=1$. Since any separably closed field of characteristic 0 contains a copy of $\bar{\mathbb{Q}}$, we obtain an inclusion $\bar{\mathbb{Q}}\hookrightarrow k $, using the last theorem again we conclude
 that $\pi_1(\mathbb{A}_k^n)=1$.
\end{proof}

\begin{exmpl}\label{E:APP}
Assume $Char(k)=p$ then the fundamental group $\pi_1(\mathbb{A}_k^n)^{ab}$ is a pro-$p$ group.
\end{exmpl}

\begin{proof}
We can show it by induction on $n$.

The case $n=1$ follows from \cite[XIII, Corollary 2.12]{SGAI}, which calculates more generally the non-$p$ part of the abelianization of the fundamental group in the affine curve case. It also follows from the Abhyankar's conjecture, which gives the finite quotients of the fundamental group of affine curves, and was proved for the affine line in \cite{Ra} and in general in \cite{Harb}.

Assume that for $n-1$ the fundamental group is a pro-$p$ group. By a basic property of the fundamental functor (see \cite{Mu1}), the fundamental group of a fiber product is given by $\pi_1(X\underset{Z}{\times}Y)\cong \pi_1(X)\underset{\pi_1(Z)}{\times}{\pi_1(Y)}$  therefore $\pi_1(\mathbb{A}_k^n)\cong \pi_1(\mathbb{A}_k^{n-1})\times \pi_1(\mathbb{A}_k^{1})$ so $\pi_1(\mathbb{A}_k^n)^{ab}\cong \pi_1(\mathbb{A}_k^{n-1})^{ab}\times \pi_1(\mathbb{A}_k^{1})^{ab}$ is a product of pro-$p$ group by the induction hypothesis, and therefore a pro-$p$ group.
\end{proof}

\begin{rem}
In the case $Char(k)=p$, the fundamental group $\pi_1(\mathbb{A}_k)^{ab}$ is non trivial, since we have the Artin-Schreier extension $k[x]\hookrightarrow k[x,y]/(y^p-y-a)$ which is finite \'etale extension, for more details consult \cite{Mi}.
\end{rem}

\begin{defn}
\begin{enumerate}
\item
A scheme $X$ over $k$ is {\em separable (over $k$)} if for any field extension $K$ of $k$, $X\underset{k}{\otimes}K$  is reduced.
\item
A morphism $f:Y\rightarrow X$ is {\em separable} if $Y$ is flat over $X$ and for any point $x\in X$ the fiber $Y\underset{X}{\otimes}k(x)$ is separable over $k(x)$.
\item
A morphism $f:Y\rightarrow X$ is a {\em fibration} if $Y$ is locally for the \'etale topology a product of the form $F\times X$ for some scheme $F$.
\end{enumerate}
\end{defn}

It is known from algebraic topology that fibrations give rise to exact sequences of the homotopy groups. We state an algebraic analog to this result:

\begin{thm}[The first homotopy exact sequence]\label{E:FHES}
Assume that $f:Y\rightarrow X$ is either a proper and separable morphism or a fibration. Assume also that $X,Y, f^{-1}(x)$ are geometrically connected. Then we have an exact sequence of the fundamental groups
\begin{center}
$\pi_1(f^{-1}(x))\rightarrow\pi_1(Y)\rightarrow\pi_1(X)\rightarrow1$.
\end{center}
\end{thm}

\begin{proof}
For the proper separable case see \cite[Chapter VI, 6.3]{Mu1}. For the case of fibration let $p:\widetilde{X}\rightarrow X$ denotes the universal covering of $X$, then we obtain a fibration with base $\widetilde{X}$:

\begin{center}
$\xymatrix{
F\ar[r]\ar@{=}[d]&
\widetilde{X}\underset{X}{\times}Y\ar[r]^{id\underset{X}{\times}f}\ar[d]&
\widetilde{X}\ar[d]^{p}
\\
F\ar[r]&
Y\ar[r]^{f}&
X
}$
\end{center}

But a fibration over a simply connected scheme is trivial, therefore $\pi_1(F)\cong \pi_1(\widetilde{X}\underset{X}{\times}Y)$. By an elementary property of the fundamental group, for a (pro-\'etale) covering $\widetilde{X}\underset{X}{\times}Y\rightarrow Y$ the map induced on the fundamental groups $\pi_1(\widetilde{X}\underset{X}{\times}Y)\rightarrow \pi_1(Y)$ makes the former a subgroup of the latter. The quotient here is the group of automorphisms of the covering $\widetilde{X}\underset{X}{\times}Y\rightarrow Y$ which is exactly the fundamental group of $X$, this is precisely the statement that there exist an exact sequence as required:
\begin{center}
$\pi_1(F)\rightarrow\pi_1(Y)\rightarrow\pi_1(X)\rightarrow1$.
\end{center}
\end{proof}

In the case of open immersions, we can say something about the map induced on the fundamental groups:

\begin{prop}\label{E:FGSM}
If $U,X$ are connected normal varieties, $i:U\hookrightarrow X$ is an inclusion of an open subscheme, then the induced map: $i_*:\pi_1(U)\rightarrow \pi_1(X)$ is surjective.
\end{prop}

\begin{proof}
\cite[IX, Corollary 5.6]{SGAI}
\end{proof}

We end this section by claiming that the fundamental group is independent of changing a "small" subscheme, the exact result
is this:

\begin{thm}\label{E:SSC2}
Let $X$ be a locally Noetherian regular scheme, $U\subseteq X$ an open subscheme and assume $X-U$ is of codimension $\geq2$. Then the morphism induced by inclusion on the fundamental groups $\pi_1(U,x_0)\xrightarrow{\approx}\pi_1(X,x_0)$ is an isomorphism.
\end{thm}

\begin{proof}
\cite[X, Corollary 3.3]{SGAI}
\end{proof}

\section{Grothendieck's sheaf-function correspondence}\label{S:GSFC}

\subsection{Representations and local systems}

Let $X$ be a connected scheme, $\bar{x}$ a geometric point. $X_{et}$ is the \'etale site on $X$ and $Sh(X_{et})$ the category of sheaves on the \'etale site.

The next step to make our statement of class field theory more geometric is the equivalence between representations of the fundamental group and local systems. Note that in this section a reference with topological meaning refers to the \'etale topology. For example:

\begin{defn}
\begin{enumerate}
\item
A {\em local system} $\mathcal{F}\in Sh(X_{et})$ on $X$ is a locally constant sheaf of abelian groups on $X$. Here, "locally constant" means that there exist an \'etale covering $\{\iota_i:U_i\rightarrow X\}_{i\in I}$ of $X$ such that $\iota_i^*\mathcal{F}$ are constant sheaves.
\item
A {\em constructible} sheaf of abelian groups $\mathcal{F}\in Sh(X_{et})$ on $X$ is a sheaf such that there exist a decomposition of $X$ into a disjoint union of (Zariski) locally closed subset $X=\coprod X_i$ such that $\mathcal{F}|_{X_i}$ is locally constant.
\end{enumerate}
\end{defn}

\begin{rem}
The results of this subsection have analogies for $\mathcal{F}$ a sheaf of $R$-modules. But for now, let $\mathcal{F}$ be a sheaf of abelian groups.
\end{rem}

The basic connection between representations of the fundamental group and local systems is:

\begin{thm}\label{E:RLSC}
The map $\mathcal{F}\mapsto\mathcal{F}_{\bar{x}}$ is an equivalence of categories between locally constant sheaves of abelian groups on $X$ with finite stalks and finite modules (or representations) of the algebraic fundamental group.
\end{thm}

\begin{proof}[Sketch of proof]
Let us give constructions in both directions which are inverses to one another.
Given a sheaf $\mathcal{F}$ we construct the corresponding representation of the fundamental group, we must show how $\pi_1(X,\bar{x})$ acts on $\mathcal{F}_{\bar{x}}\cong G$. Let $\widetilde{X}=(X_i)_{i\in I}$ be the universal covering of $X$, where $p_i:X_i\rightarrow X$ are Galois coverings.

The local system $\mathcal{F}$ must be trivial on the universal covering because it is simply connected: it has no nontrivial \'etale coverings to trivialize the local system, but we know that the local system can be trivialize using some covering (it is locally constant). So write $p:\widetilde{X}\rightarrow X$ (this map is only "pro-\'etale", because it is not necessarily locally of finite type) and choose an isomorphism $\alpha:p^*\mathcal{F}\xrightarrow{\approx} G$ where $G$ is the constant sheaf with stalk $G$ ($G$ an abelian group). This isomorphism is determined uniquely up to a choice of isomorphism of groups $(p^*\mathcal{F})_{\bar{z}}\xrightarrow{\approx} G$ for some $\bar{z}$ over $\bar{x}$. Note that $\alpha$ defines isomorphism of groups $\alpha_{\bar{z}'}:(p^*\mathcal{F})_{\bar{z}'}\xrightarrow{\approx} G$ for any point $\bar{z}'$ of $\widetilde{X}$.

Let $g\in \pi_1(X,\bar{x})$, and let $\bar{z}'=g\bar{z}$ be two points of the universal cover over $\bar{x}$. Then we can define $\rho(g)$ to be the $G$-automorphism  $\alpha_{\bar{z}}\circ g_{\bar{z},\bar{z}'}\circ\alpha_{\bar{z}'}^{-1}\in Aut(G)$ where $g_{\bar{z},\bar{z}'}:(p^{-1}\mathcal{F})_{\bar{z}'}\xrightarrow{\approx} (p^{-1}\mathcal{F})_{\bar{z}}$ is the map induced on the stalks by $g:\widetilde{X}\rightarrow \widetilde{X}$.

Concerning the other direction, given a representation $\rho:\pi_1(X,\bar{x})\rightarrow Aut(G)$ and an \'etale morphism $U\rightarrow X$ we can define a local system on $X$. Let $\widetilde{U}$ denotes the product $\widetilde{X}\underset{X}{\times}U$, $\mathcal{F}(U)$ will be constructed as the set of all functions $f:\widetilde{U}\rightarrow G$ satisfying $f(g\bar{z})=\rho(g)f(\bar{z})$ where $g\in \pi_1(X,\bar{x})$ is an element of the fundamental group which acts on the fiber of $p$, i.e on the first factor of $\widetilde{U}$.

\end{proof}

We want to characterize local systems which are pullbacks of some other local system:

\begin{defn}\label{E:ESD}
Let $G$ be an algebraic group (in our case $G$ will be finite) acting from the left on the scheme $Y$ and let
\begin{center}
$s:Y\rightarrow G\times Y, \,p_i:G\times Y\rightarrow Y, \, i=1,2, \,q_i:G\times G\times Y\rightarrow G\times Y, \,i=1,2,3$
\end{center}
defined for $y\in Y$, $g,g'\in G$ by:
\begin{center}
$s(y)=(1,y), \,p_1(g,y)=g^{-1}y, \,p_2(g,y)=y$ $q_1(g,g',y)=(g',g^{-1}y), \,q_2(g,g',y)=(gg',y), \,q_3(g,g',y)=(g,y)$
\end{center}
A $G$-equivariant sheaf on $Y$ is a sheaf $\mathcal{F}$ on $Y$ and an isomorphism $\alpha:p_2^*\mathcal{F}\xrightarrow{\approx} p_1^*\mathcal{F}$ satisfying the cocycle conditions:
\begin{center}
$s^*(\alpha)=id_{\mathcal{F}}, \,q_2^*(\alpha)=q_1^*(\alpha)\circ q_3^*(\alpha)$
\end{center}
\end{defn}

\begin{prop}\label{E:EVLS}
Let $\pi:Y\rightarrow X$ be a Galois covering with Galois group $G=Aut_X(Y)$, $\bar{x},\bar{y}$ as above. For a local system $\mathcal{F}$ on $Y$, the following are equivalent:

\begin{enumerate}
\item There exist a $G$-equivariant structure on $\mathcal{F}$, $\alpha:p_2^*\mathcal{F}\rightarrow p_1^*\mathcal{F}$.
\item $\mathcal{F}=\pi^*\mathcal{F}'$ for some local system $\mathcal{F}'\in Sh(X_{et})$.
\end{enumerate}
\end{prop}

\begin{proof}
$(1)\Rightarrow(2)$: $\mathcal{F}$ is $G$-equivariant if there exists an isomorphism $\alpha:p_2^*\mathcal{F}\rightarrow p_1^*\mathcal{F}$ satisfying the cocycle condition. The statement that $\pi:Y\rightarrow X$ is Galois covering means that the map $f:G\times Y \rightarrow Y\underset{X}{\times}Y$ defined by $f(g,y)=(g^{-1}y,y)$ is an isomorphism.

Let $\pi_i:Y\underset{X}{\times}Y\rightarrow Y$ $i=1,2$ be the two projections. It is clear that  $p_i=f\circ \pi_i$ $i=1,2$ therefore, since $f$ is an isomorphism, the isomorphism $\alpha:p_2^*\mathcal{F}\rightarrow p_1^*\mathcal{F}$ gives after base change under $f^{-1}$ an isomorphism $\beta=(f^{-1})^*\alpha:\pi_2^*\mathcal{F}\rightarrow \pi_1^*\mathcal{F}$.

We can finish the proof using the gluing property for sheaf on $X$. Let $(f_i:U_i\rightarrow X)$ be a covering of $X$ where $f_i$ are \'etale morphisms of finite type, and let $\pi_1:U_i\underset{X}{\times}U_j\rightarrow U_i$, $\pi_2:U_i\underset{X}{\times}U_j\rightarrow U_j$ denotes the projections. The gluing property of sheaves states that to give a sheaf  on $X$ is the same as to give sheaves $\mathcal{F}_i$ on $U_i$ and isomorphisms of sheaves on $U_i\underset{X}{\times}U_j$: $\pi_1^*\mathcal{F}_i\xrightarrow{\approx} \pi_2^*\mathcal{F}_j$ for all $i,j$.

In our case, the covering is simply $\pi:Y\rightarrow X$, then (2) is equivalent to an isomorphism over $Y\underset{X}{\times}Y$ between $\pi_1^*\mathcal{F}\xrightarrow{\approx}\pi_2^*\mathcal{F}$ but $\beta$ is exactly such an isomorphism.

$(2)\Rightarrow(1)$: The same proof, in the reverse direction, works here too. Note that the cocycle condition for equivariant sheaves becomes trivial. Indeed, the maps $\sigma:Y\underset{X}{\times}Y\xrightarrow{\pi_i}Y\xrightarrow{\pi}X$ coincide, and if $\mathcal{F}=\pi^*\mathcal{F}'$ then the isomorphism $\beta:\pi_2^*\mathcal{F}\rightarrow \pi_1^*\mathcal{F}$ is just the identity on $\sigma^*\mathcal{F}'$. Therefore $\alpha:p_2^*\mathcal{F}\rightarrow p_1^*\mathcal{F}$ is also the identity and the cocycle condition follows immediately.
\end{proof}

\begin{rem}
If we denote the category of $G$-equivariant sheaves on $Y$ by $Sh_G(Y_{et})$ then
$\pi^*:Sh(X_{et})\rightarrow Sh_G(Y_{et})$ is an equivalence of categories. Its inverse is $(\pi_*-)^G:Sh_G(Y_{et})\rightarrow Sh(X_{et})$ defined by $\mathcal{F}\mapsto (\pi_*\mathcal{F})^G$, the invariants of the direct image. See \cite{Vis} for treatment of the subject in a more general context.
\end{rem}

\subsection{$\ell$-adic local systems}

We first recall here some definitions concerning sheaves:

\begin{defn}
Let $\mathbb{Q}_{\ell}\subseteq F$ be a finite extension, $\mathcal{O}_F$ the ring of integers in $F$ and $\mathfrak{m}_F$ its maximal ideal.
\begin{enumerate}
\item
An {\em $\mathcal{O}_F$-sheaf} is a projective system $(\mathcal{F}_n)_{n\in\mathbb{N}}$ where $\mathcal{F}_n$ is a constructible sheaf of $\mathcal{O}_F/\mathfrak{m}_F^n$ modules such that $\mathcal{F}_{n+1}\rightarrow\mathcal{F}_n$ induces isomorphisms $\mathcal{F}_{n+1}\underset{\mathcal{O}_F/\mathfrak{m}_F^{n+1}}{\otimes}\mathcal{O}_F/\mathfrak{m}_F^n\xrightarrow{\approx}\mathcal{F}_n$.
\item
The category of {\em $F$-sheaves} is the quotient of the category of {\em $\mathcal{O}_F$-sheaf} by the torsion objects (sheaves that killed by some power of $\ell$. Specifically, these are sheaves $(\mathcal{F}_n)_{n\in\mathbb{N}}$ such that for $n$ large enough $\mathcal{F}_n=\mathcal{F}$ is independent of $n$). Given an $\mathcal{O}_F$-sheaf $\mathcal{F}$ its image under this quotient is denoted by $F\otimes_{\mathcal{O}_F}\mathcal{F}$.
\item
If $\mathbb{Q}_{\ell}\subseteq F \subseteq F'$ are finite extensions. Then there exists a functor from the category of $F$-sheaves to $F'$-sheaves $\mathcal{F}\mapsto \mathcal{F}\underset{F}{\otimes}F'$ defined by extension of coefficients. A {\em $\bar{\mathbb{Q}}_{\ell}$-sheaf} is a direct limit $\varinjlim \mathcal{F}_F$, where $\mathcal{F}_F$ is a $F$-sheaf and $F$ runs through the finite extensions of $\mathbb{Q}_{\ell}$.
\end{enumerate}
\end{defn}

\begin{defn}
Let $\mathbb{Q}_{\ell}\subseteq F$ be a finite extension, $\mathcal{O}_F$ the ring of integers in $F$ and $\mathfrak{m}_F$ its maximal ideal.
\begin{enumerate}
\item
An $\mathcal{O}_F$-sheaf $(\mathcal{F}_n)_{n\in\mathbb{N}}$ is {\em locally constant} if each $\mathcal{F}_n$ is.
\item
A $F$-sheaf $F\otimes_{\mathcal{O}_F}\mathcal{F}$ is a {\em locally constant} is $\mathcal{F}$ is.
\item
A $\bar{\mathbb{Q}}_{\ell}$-sheaf is {\em locally constant}, if it can be written as a limit of locally constant $F$-sheaves. Such a sheaf is called an {\em $\ell$-adic local system}.
\end{enumerate}
\end{defn}

\begin{rem}
The definitions of direct and inverse images of sheaves and the statements of the previous subsection have analogs in the $\ell$-adic case. The process is the following: First, take inverse limit to get versions for $\mathcal{O}_F$-sheaves. Then tensor with $F$ to get rid of the torsion and state version for $F$ sheaves. Finally, by direct limit conclude the corresponding statements about $\ell$-adic local systems. We immediately state such a version of theorem \re{RLSC}.
\end{rem}

\begin{thm}\label{E:LSRL}
Let $X$ be a connected scheme, $\bar{x}$ a geometric point, then there is an equivalence of categories between
$\ell$-adic local systems on $X$ and continuous representations of $\pi_1(X,\bar{x})$ in finite dimensional $\bar{\mathbb{Q}}_{\ell}$ vector spaces. It is defined by $\mathcal{F}\mapsto\mathcal{F}_{\bar{x}}$
\end{thm}

\begin{rem}
\begin{enumerate}
\item
In particular we get a correspondence between one dimensional $\ell$-adic local systems and characters.
\item
We can use the theorem above to apply notions defined for representations also for local systems and conversely. In particular, the notion of ramification of a local system is defined to be the ramification of the corresponding representation.
\end{enumerate}
\end{rem}

\subsection{The Frobenius action}

In this subsection, $X$ will be an algebraic variety over the finite field $\mathbb{F}_q$.

\begin{note}
$\bar{X}$ will denote $X\underset{Spec\,\mathbb{F}_q}{\times}Spec\,\bar{\mathbb{F}}_q$, it is an algebraic variety over $\bar{\mathbb{F}}_q$.

The projection map $\bar{X}\rightarrow X$ will be denoted by $\pi_X$.

We will use the same notation for sheaves, if $\mathcal{F}$ is a sheaf over $X$, we will write $\bar{\mathcal{F}}$ for the pullback $\pi^*_X\mathcal{F}$.
\end{note}

Now we define three different endomorphisms of $\bar{X}$:

\begin{defn}
\begin{enumerate}
\item
The {\em absolute Frobenius} $Ab_X:X\rightarrow X$ is defined as the identity on the underlying space and $f\mapsto f^q$ on the structure sheaf $\mathcal{O}_{\bar{X}}$, similarly we can define $Ab_{\bar{X}}:\bar{X}\rightarrow \bar{X}$.

\item
The {\em geometric Frobenius} map $F_X:\bar{X}\rightarrow \bar{X}$ is the map $Ab_X\times id_{Spec\,\bar{\mathbb{F}}_q}: X\underset{Spec\,\mathbb{F}_q}{\times}Spec\,\bar{\mathbb{F}}_q \rightarrow X\underset{Spec\,\mathbb{F}_q}{\times}Spec\,\bar{\mathbb{F}}_q$.
\item
The {\em arithmetic Frobenius} map $F_X:\bar{X}\rightarrow \bar{X}$ is the map $id_X\times Ab_{Spec\bar{\mathbb{F}}_q}: X\underset{Spec\,\mathbb{F}_q}{\times}Spec\,\bar{\mathbb{F}}_q \rightarrow X\underset{Spec\,\mathbb{F}_q}{\times}Spec\,\bar{\mathbb{F}}_q$.

\end{enumerate}
\end{defn}

\begin{note}
We sometimes omit $X$ from the notation when it is clear from the context and write simply $F$, $Ar$ and $Ab$ for $F_X$, $Ar_X$ and $Ab_X$ respectively.
\end{note}

\begin{rem}\label{E:FAF}
\
\begin{enumerate}
\item 
The geometric Frobenius $F_X:\bar{X}\rightarrow \bar{X}$ is a morphism over $\bar{\mathbb{F}}_q$ and $Ar_X:\bar{X}\rightarrow \bar{X}$ is an automorphism of $\bar{X}$ corresponding to the Frobenius automorphism of the field extension $\bar{\mathbb{F}}_q/\mathbb{F}_q$.

\item
From the definition it is clear that the following relation between the Frobenii is satisfied:
\begin{center}
$Ar_X\circ F_X =F_X \circ Ar_X = Ab_X \times Ab_{Spec\bar{\mathbb{F}}_q} = Ab_{\bar{X}}$
\end{center}

\item
The Frobenii are functorial. If $\theta:X\rightarrow Y$ is a map of schemes over $\mathbb{F}_q$ then $Ab_Y\circ \theta= \theta\circ Ab_X$. The equality as maps of the underlying spaces is clear, and the equality of the maps on the structure sheaves also follows since the map $g:\mathcal{O}_Y(U)\rightarrow \theta_*\mathcal{O}_X(U)$ is an homomorphism of rings, and therefore for $f\in \mathcal{O}_Y(U)$ we have: $g(f^q)=g(f)^q$. We conclude that if $\bar{\theta}:\bar{X}\rightarrow \bar{Y}$ is the induced map of schemes over $\bar{\mathbb{F}}_q$, then $F_Y\circ \bar{\theta}=(Ab_Y\times id_{Spec\,\bar{\mathbb{F}}_q}) \circ (\theta \times id_{Spec\,\bar{\mathbb{F}}_q}) = (\theta \times id_{Spec\,\bar{\mathbb{F}}_q}) \circ (Ab_X\times id_{Spec\,\bar{\mathbb{F}}_q})=\bar{\theta} \circ F_X$ therefore the geometric Frobenius is functorial, and similarly for the arithmetic Frobenius $Ar_Y\circ \bar{\theta}=\bar{\theta} \circ Ar_X$.
\end{enumerate}
\end{rem}

\begin{defn}
A {\em Weil sheaf} $\mathcal{F}$ over $X$ is a pair $(\bar{\mathcal{F}},\psi)$ where $\bar{\mathcal{F}}$ is an $\ell$-adic sheaf over $\bar{X}$ and $\psi$ is an isomorphism of sheaves: $\psi:F^*\bar{\mathcal{F}}\xrightarrow{\approx}\bar{\mathcal{F}}$. We can  think about $\psi$ as an action of Frobenius.
\end{defn}

\begin{con}\label{E:SFCFQ}
Let $\mathcal{F}=(\bar{\mathcal{F}},\psi)$ is a Weil sheaf on $Spec\,\mathbb{F}_q$. $\bar{\mathcal{F}}$ is a sheaf on $Spec\,\bar{\mathbb{F}}_q$ so it is just a $\bar{\mathbb{Q}}_{\ell}$ vector space (corresponding to the stalk in the zero ideal of $\bar{\mathbb{F}}_q$). The geometric Frobenius $F$ is the identity morphism in this case, therefore $\psi:\bar{\mathcal{F}}\rightarrow \bar{\mathcal{F}}$ is an endomorphism of this vector space, whose trace will be denoted by $Tr(\mathcal{F})$.
\end{con}

\begin{prop}
For any sheaf $\bar{\mathcal{F}}$ on $\bar{X}$ there exist a canonical isomorphism: $\theta:Ab^*\bar{\mathcal{F}}\xrightarrow{\approx}\bar{\mathcal{F}}$.
\end{prop}

\begin{proof}
\cite[XV, Expos\'e $\S2$, $n^{\circ}1$]{SGA5}.
\end{proof}

We now give a structure of Weil sheaves to sheaves that come from $X$:

\begin{prop}\label{E:FQWS}
If $\mathcal{F}$ is a sheaf on $X$, then $\bar{\mathcal{F}}=\pi_X^*\mathcal{F}$ has a natural action of Frobenius that makes it a Weil sheaf.
\end{prop}

\begin{proof}
By the above proposition we have: $(F\circ Ar)^*\bar{\mathcal{F}}=Ab^*\bar{\mathcal{F}}\cong\bar{\mathcal{F}}$ apply $(Ar^*)^{-1}$ on this equation to get $F^*\bar{\mathcal{F}}\cong (Ar^*)^{-1}\bar{\mathcal{F}}$. Since $Ar$ acts only on the second factor of the structure sheaf of $\bar{X}=X\underset{Spec\,\mathbb{F}_q}{\times}Spec\,\bar{\mathbb{F}}_q$ we have $\pi_X\circ Ar=\pi_X$ and therefore:
\begin{center}
$F^*\bar{\mathcal{F}}\cong (Ar^*)^{-1}\bar{\mathcal{F}}=
(Ar^*)^{-1}\pi_X^* \mathcal{F}=\pi_X^* \mathcal{F}=\bar{\mathcal{F}}$
\end{center}
Is the required action of Frobenius.
\end{proof}

\begin{prop}\label{E:PBWS}
Let $\theta:X\rightarrow Y$ be a map, and let $(\bar{\mathcal{F}},\psi)$ be a Weil sheaf on $Y$. Then the pullback $\bar{\theta}^*\bar{\mathcal{F}}$ has a natural action of Frobenius that makes it a Weil sheaf.
\end{prop}

\begin{proof}
We have the isomorphism $\psi:F_Y^*\bar{\mathcal{F}}\xrightarrow{\approx} \bar{\mathcal{F}}$. Apply $\bar{\theta}^*$ to get  $\bar{\theta}^*\psi:\bar{\theta}^*F_Y^*\bar{\mathcal{F}}\xrightarrow{\approx} \bar{\theta}^*\bar{\mathcal{F}}$. But by functoriality of the Frobenius, $\bar{\theta}\circ F_X = F_Y \circ \bar{\theta}$ therefore $\bar{\theta}^*\psi:F_X^*\bar{\theta}^*\bar{\mathcal{F}}\xrightarrow{\approx} \bar{\theta}^*\bar{\mathcal{F}}$ is an action of Frobenius on $\bar{\theta}^*\bar{\mathcal{F}}$.
\end{proof}

\begin{rem}\label{E:TPWS}
If $(\bar{\mathcal{F}}_1,\psi_1), (\bar{\mathcal{F}}_2,\psi_2)$ are Weil sheaves on $X$ then the tensor product $\bar{\mathcal{F}}_1\otimes \bar{\mathcal{F}}_2$ has also a structure of a Weil sheaf by: $F^*(\bar{\mathcal{F}}_1\otimes \bar{\mathcal{F}}_2)\xrightarrow{\approx} F^*\bar{\mathcal{F}}_1\otimes F^*\bar{\mathcal{F}}_2\xrightarrow{\psi_1\otimes \psi_2}\bar{\mathcal{F}}_1\otimes \bar{\mathcal{F}}_2$.
\end{rem}

\begin{con}[Sheaf-function correspondence]\label{E:SFC}
Let $X(\mathbb{F}_q)$ be the set of $\mathbb{F}_q$ points of $X$. To a Weil sheaf $\mathcal{F}=(\bar{\mathcal{F}},\psi)$ we can associate a function $f^{\mathcal{F}}:X(\mathbb{F}_q)\rightarrow \bar{\mathbb{Q}}_{\ell}$ defined by:
\begin{center}
$f^{\mathcal{F}}(x)=Tr(i_x^*\mathcal{F})$
\end{center}
Where $i_x:Spec\,\mathbb{F}_q\rightarrow X$ is just the $\mathbb{F}_q$-point $x$.
\end{con}

\begin{rem}\label{E:RSFC}
This notion is convenient since many useful functions come from sheaves, and operations on functions correspond to operations on sheaves. For example, let $g:X\rightarrow Y$ then:
\begin{enumerate}
\item
Pullback of functions and sheaves correspond to one another, if $\mathcal{F}$ is a Weil sheaf on $Y$ then:
\begin{center}
$f^{g^*(\mathcal{F})}=g^*(f^{\mathcal{F}})$
\end{center}
This property is immediate from the definition, since if $x\in X(\mathbb{F}_q)$ maps to $g(x)=y\in Y(\mathbb{F}_q)$, then $i_y^*\mathcal{F}=i_x^*g^*\mathcal{F}$ since $i_y=g\circ i_x$.
\item
Tensor product of sheaves corresponds to product of functions, if $\mathcal{F}_1,\mathcal{F}_2$ are Weil sheaves on $X$:
\begin{center}
$f^{\mathcal{F}_1\otimes\mathcal{F}_2}=f^{\mathcal{F}_1}\cdot f^{\mathcal{F}_2}$
\end{center}
This is also trivial using the definition of the Weil sheaf structure on the tensor product and the fact that if $f_i:V_i\rightarrow V_i$ $i=1,2$ are linear maps of the vector spaces $V_i$ then $Tr(f_1\otimes f_2)=Tr(f_1)Tr(f_2)$
\end{enumerate}

\end{rem}
More importantly, sheaves have other operations that don't exist for functions, making them better objects to work with.

\section{Proof of the first part of class field theory}\label{S:PFPCFT}

\subsection{Unramified case}

Let $X$ be a smooth projective connected curve defined over $\mathbb{F}_q$ with function field $K=\mathbb{F}_q(X)$ we have already proved that $Pic_D(X)(\mathbb{F}_q)\cong K^{\times}\backslash\mathbb{A}_K^{\times}/\mathbb{O}_D^{\times}$ (\re{GPGI}). The following statement will be useful:

\begin{prop}\label{E:TLI}
Let $H$ be a commutative algebraic group defined over $\mathbb{F}_q$, then there exists a surjective morphism $L_H:\pi_1(H)\twoheadrightarrow H(\mathbb{F}_q)$ defined in the proof.
\end{prop}

\begin{proof}
The {\em Lang isogeny} $L:H\rightarrow H$ is defined by $a\mapsto F(a)/a$, where $F:H\rightarrow H$ is the geometric Frobenius defined above. $L$ is a Galois covering with group of covering transformation $ker(L)=H(\mathbb{F}_q)$ (It acts on the covering by translation, $T_x:H\xrightarrow{\approx} H$ is $T_x(y)=xy$). But $\pi_1(H)$ is the group of covering transformations of the universal covering, therefore $H(\mathbb{F}_q)$ is a quotient $L_H:\pi_1(H)\twoheadrightarrow H(\mathbb{F}_q)$.
\end{proof}

\begin{rem}\label{E:RXSF}
\
\begin{enumerate}
\item
A character $\xi:H(\mathbb{F}_q)\rightarrow \bar{\mathbb{Q}}^{\times}_{\ell}$ defines by composition with $L_H$ a character $\pi_1(H)\rightarrow  \bar{\mathbb{Q}}^{\times}_{\ell}$, which is equivalent to a one dimensional local system $\mathcal{H}$ on $H$ by theorem \re{RLSC}. Since $\mathcal{H}$ is defined over $\mathbb{F}_q$ it has a structure of a Weil sheaf by proposition \re{FQWS}. We claim that $f^{\mathcal{H}}=\xi$.

Indeed, let $x\in H(\mathbb{F}_q)$ and we use here the notations of the sheaf function correspondence \re{SFC}. The value of the function $f^{\mathcal{H}}(x)$ is defined to be the trace of the action of Frobenius on the stalks of $i_x^*\mathcal{H}$, in the one dimensional case, the action of Frobenius is simply multiplication by an element of $\bar{\mathbb{Q}}^{\times}_{\ell}$, hence the trace is equal to this element, so by definition it is  $f^{\mathcal{H}}(x)$.

On the other hand if $y$ is a point such that $L(y)=x$ then $x=F(y)/y$ hence $xy=F(y)$, therefore the action of Frobenius on $i_x^*\mathcal{H}$ is induced by the action of $T_x:H\rightarrow H$ defined by $T_x(y)=xy$ on the covering $L:H\rightarrow H$. But the action of $T_x$ is exactly how $x\in H(\mathbb{F}_q)$ acts on the covering, which is multiplication by $\xi(x)$. We conclude that $f^{\mathcal{H}}(x)=\xi(x)$.

\item
If $\mathcal{F}$ is a one dimensional local system on $X$ and $\rho: \pi_1(X)\rightarrow \bar{\mathbb{Q}}^{\times}_{\ell}$ the corresponding representation under the equivalence of theorem \re{LSRL}, $\mathcal{F}$ defines a Weil sheaf by proposition \re{FQWS}. Let $x\in X$ be the $\mathbb{F}_q$ point of $X$ corresponds to the valuation $v$ of $K$, then $\rho(Fr_v)=f^{\mathcal{F}}(x)$. Indeed, consider the map $i_x$ defined by composition $Spec\,\mathbb{F}_q\rightarrow Spec\,\mathcal{O}_x\rightarrow X$. The inverse image of $Fr_v$ under the map $\pi_1(Spec\,\mathbb{F}_q)\rightarrow \pi_1(X)$ is  $x\mapsto x^q\in Gal(\bar{\mathbb{F}}_q/\mathbb{F}_q)$ hence the pullback $i_x^*\rho$ is defined on the generator by $i_x^*\rho(x\mapsto x^q)=\rho(Fr_v)$. But $x\mapsto x^q$ is simply the geometric Frobenius on the structure sheaf of $Spec\,\mathbb{F}_q$. Therefore by definition $f^{\mathcal{F}}(x)=Tr(i_x^*\mathcal{F})=\rho(Fr_v)$.
\end{enumerate}
\end{rem}

\begin{proof}[Proof of existence \re{UCFT} (1)]
We start from a character $\xi: K^{\times}\backslash\mathbb{A}_K^{\times}/\mathbb{O}_D^{\times}\cong Pic_D(\mathbb{F}_q)\rightarrow \bar{\mathbb{Q}}_{\ell}^{\times}$. By remark (1) above it gives a Weil sheaf $\mathcal{H}$ on $Pic_D$. We pull $\mathcal{H}$ back under the map $\phi:X-D\rightarrow Pic_D$ defined above in \re{PHIDEF} to get a Weil sheaf $\mathcal{F}$ on $X-D$, therefore (using theorem \re{LSRL}) we get a representation $\rho: \pi_1(X-D)\rightarrow \bar{\mathbb{Q}}^{\times}_{\ell}$. In the case $[D]=0$ this is a character of $\pi_1(X)$ which is simply an unramified character of $G_K$.

Now, it remains to show that $\rho(Fr_v)=\xi(\pi_v)$ for an unramified prime $v$ of $K$, (this proof will also work for the ramified case). Let $x$ be the $\mathbb{F}_q$-point of $X$ corresponding to $v$, by remark (2) above, $\rho(Fr_v)=f^{\mathcal{F}}(x)$. Since $\phi(x)=\mathcal{O}([x])$ we get by remark \re{RSFC} (1) that $f^{\mathcal{F}}(x)=f^{\mathcal{H}}(\mathcal{O}([x]))$. But $[x]$ corresponds to $\pi_v$ under the correspondence between $Cl_D$ and $K^{\times}\backslash\mathbb{A}_K^{\times}/\mathbb{O}^{\times}_D$, hence by remark (1) above $f^{\mathcal{H}}(\mathcal{O}([x]))=\xi(\pi_v)$ and we are done.
\end{proof}

\subsection{Tamely ramified case}

\begin{proof}[Proof of existence in \re{RCFT} (1), tame ramification case]
Here $[D]=\underset{i}{\sum} x_i$ is a sum of points with multiplicity 1. Let $\xi:Pic_D(\mathbb{F}_q)\rightarrow \bar{\mathbb{Q}}_{\ell}^{\times}$ be a character, and consider also $\xi^{q-1}$. In our case, proposition \re{FPIC} gives a short exact sequence of schemes, since $Hom(Spec\,\mathbb{F}_q,-)$ is left exact we obtain the sequence:

\begin{center}
$1\rightarrow(\underset{i}{\prod}\mathbb{G}_m)/\mathbb{G}_m(\mathbb{F}_q)\xrightarrow{f}Pic_D(\mathbb{F}_q)\xrightarrow{g} Pic(\mathbb{F}_q)\rightarrow Ext^1(Spec\,\mathbb{F}_q, (\underset{i}{\prod}\mathbb{G}_m)/\mathbb{G}_m)$
\end{center}
Where we wrote the kernel $R_D\cong \underset{i}{\prod}\mathbb{G}_m)/\mathbb{G}_m$ explicitly in the tamely ramified case.
By Hilbert theorem 90 we have $Ext^1(Spec\,\mathbb{F}_q,\mathbb{G}_m)=1$, and by properties of $Ext$, the last element in the sequence above is also trivial.

But $\mathbb{G}_m(\mathbb{F}_q)$ is of order $q-1$, hence $f^*(\xi^{q-1})=(f^*(\xi))^{q-1}$ is trivial on the kernel, so it induces a character of $\eta:Pic(\mathbb{F}_q)\rightarrow \bar{\mathbb{Q}}_{\ell}^{\times}$. By the unramified case proved above, we obtain from $\eta$ an unramified representation $\tau:G_K\rightarrow\bar{\mathbb{Q}}_{\ell}^{\times}$. By the construction of the above subsection, $\xi$ itself gives a representation of $\rho:G_K\rightarrow \bar{\mathbb{Q}}_{\ell}^{\times}$, and we want to bound the ramification of $\rho$. In the tamely ramified case, the theorem states that $\rho$ should be trivial on the ramification groups $G^j_{K_{x_i}}$ for $j\geq1$. We already know from the unramified case, that $\tau=\rho^{q-1}$ is trivial on $G^j_{K_{x_i}}$ for $j\geq0$. But by proposition \re{PHRG}, $G^j_{K_{x_i}}$ for $j\geq1$ are pro-$p$ groups and since $q-1$ is prime to $p$, if $\rho^{q-1}$ is trivial on these group, then so is $\rho$.
\end{proof}

\section{The main theorem}\label{S:TMT}

Here, $X$ is a smooth projective connected curve over a separably closed field $k$. In this section we state the main theorem of this paper, and show why it completes the proof of class field theory. Recall the maps $\phi:X-D\rightarrow Pic_D, act:(X-D)\times Pic_D\rightarrow Pic_D$ from definition \re{PHIDEF} and define:

\begin{defn}
Let $\mathcal{F}$ be a sheaf on $X-D$. A {\em Hecke eigensheaf} on $Pic_D$ with eigenvalue $\mathcal{F}$ is a pair $(E,i)$, where $E$ is a sheaf on $Pic_D$ and $i:act^*E\xrightarrow{\approx}\mathcal{F}\boxtimes E$ is an isomorphism.

An isomorphism of Hecke eigensheaves $(E_1,i_1)$ and $(E_2,i_2)$ is an isomorphism $j:E_1\xrightarrow{\approx}E_2$ such that the following diagram commutes:
\begin{center}
$\begin{CD}
act^*E_1 @>i_1>\approx>\mathcal{F}\boxtimes E_1\\
@V\approx Vact^*jV @V \approx Vid_{\mathcal{F}}\boxtimes jV\\
act^*E_2 @>i_2>\approx>\mathcal{F}\boxtimes E_2
\end{CD}$
\end{center}
\end{defn}

In the unramified case, the statement of the main theorem is the following:

\begin{thm}\label{E:MTU}[The main theorem - unramified case]
For any $\ell$-adic local system of rank 1 $\mathcal{F}$ on $X$, there exists a Hecke eigensheaf $(E_{\mathcal{F}},i_{\mathcal{F}})$ with eigenvalue $\mathcal{F}$ on $Pic$ satisfying $\phi^*E_{\mathcal{F}}=\mathcal{F}$ , it is unique up to unique isomorphism (here $act$ is the same as above, but in the case $[D]=0$).
\end{thm}

In the ramified case, the statement is a little more subtle, and we have to distinguish between two cases, according to the characteristic of $k$. Only the case of $Char(k)=p>0$ will be used to prove class field theory and we will prove it in the tamely ramified case, but since also proving the characteristic zero version doesn't add any new difficulties, we will prove it too. But first, let us state them:

\begin{thm}\label{E:MTR0}[The main theorem - ramified case, $Char(k)=0$]
For any $\ell$-adic local system of rank 1 $\mathcal{F}$ on $X-D$, there exists a Hecke eigensheaf $(E_{\mathcal{F}},i_{\mathcal{F}})$ with eigenvalue $\mathcal{F}$ on $Pic_D$ satisfying $\phi^*E_{\mathcal{F}}=\mathcal{F}$, it is unique up to unique isomorphism.
\end{thm}

\begin{thm}\label{E:MTRP}[The main theorem - ramified case, $Char(k)=p$]
For any $\ell$-adic local system of rank 1 $\mathcal{F}$ on $X-D$ with $ram(\mathcal{F})\leq [D]$, there exists a Hecke eigensheaf $(E_{\mathcal{F}},i_{\mathcal{F}})$ with eigenvalue $\mathcal{F}$ on $Pic_D$ satisfying $\phi^*E_{\mathcal{F}}=\mathcal{F}$, it is unique up to unique isomorphism.
\end{thm}

\begin{prop}\label{E:GHEP}
A more general property than the Hecke eigensheaf property actually holds: if $m:Pic_D\times Pic_D\rightarrow Pic_D$ is the group operation: $(\mathcal{L},\mathcal{L}')\mapsto \mathcal{L}\otimes\mathcal{L}'$ then $m^*E_{\mathcal{F}}\cong E_{\mathcal{F}} \boxtimes E_{\mathcal{F}}$.
\end{prop}

\begin{rem}
In the characteristic zero version, note the absence of a condition about the ramification of the local system $\mathcal{F}$. It implies that the condition on $\mathcal{F}$ doesn't depend on $D$ but only on $Supp(D)$; inspired by this, we can conjecture (and prove) that $\pi_1(Pic_D)$ depends only on $Supp(D)$ too. This is indeed the case.
\end{rem}

\begin{prop}\label{E:CHAR0}
Assume that $Supp(D)=Supp(D')$ are two subschemes and $Char(k)=0$, then $\pi_1(Pic_D)\cong \pi_1(Pic_{D'})$.
\end{prop}

\begin{proof}
Since we can pass through any two such divisors by a finite sequence of adding and deleting a point $[x]$, without loss of generality it is enough to prove the theorem for the case $[D']=[D]+[x]$. Where, of course, $x\in Supp(D)$.

If we use the first exact sequence of homotopy \re{FHES} and the fact that the affine space (in characteristic zero) is simply connected \re{ASC}, it is enough to prove that the kernel of the map $Pic_{D'}\rightarrow Pic_D$ is $\mathbb{A}^1$. We consider this map as a map of groups, though it is actually a map of algebraic groups and to show it precisely the functorial approach must be used as we already stated above in proposition \re{FPIC} and its remark.

Write $[D]=\sum n_ix_i+nx$, the map $g:Pic_{D'}\rightarrow Pic_D$ is defined to be $g(\mathcal{L},\psi)=(\mathcal{L},\psi|_D)$, therefore an element is in the kernel if $\mathcal{L}\cong\mathcal{O}$ and the restriction of the trivialization to $D$ is the identity.

A trivialization over $D'$ is an automorphism of  $\mathcal{O}_x/\mathfrak{m}_x^{n+1}\times \prod \mathcal{O}_{x_i}/\mathfrak{m}_{x_i}^{n_i}$ preserving each factor and we are looking for trivializations such that the restriction to $D$ is the identity, i.e. the automorphism is trivial when restricted to $\mathcal{O}_x/\mathfrak{m}_x^{n}\times \prod \mathcal{O}_{x_i}/\mathfrak{m}_{x_i}^{n_i}$. Such trivializations are addition of elements from $a\in \mathfrak{m}_x^{n}$: $b\in \mathcal{O}_x/\mathfrak{m}_x^{n+1}\mapsto a+b\in\mathcal{O}_x/\mathfrak{m}_x^{n+1}$ where we ignore the rest of the factors which are irrelevant for this discussion. Of course, the automorphism $a$ is trivial iff $a\in \mathfrak{m}_x^{n+1}$, therefore the kernel is simply $\mathfrak{m}_x^{n}/\mathfrak{m}_x^{n+1}$ but since $\mathcal{O}_x$ is a DVR we get that the latter is a one dimensional $k$ vector space so it is non-canonically isomorphic (as an additive group) to $k$, which can be identified with $\mathbb{A}^1$.

\end{proof}

\begin{proof}[Proof of existence in theorem \re{RCFT} (2)]
Let $\rho:G_K\rightarrow\bar{\mathbb{Q}}_{\ell}^{\times}$ be a continuous character with $ram(\rho)\leq [D]$ as given in the theorem. Here $X$ is a smooth projective connected curve defined over $\mathbb{F}_q$ with function field $K$. $\rho$ gives rise to a representation of $\pi_1(X-D)$,  with ramification bounded by $[D]$. By theorem \re{LSRL}, it is the same as $\ell$-adic local system $\mathcal{F}_{\rho}$ on $X-D$ with ramification bounded by $[D]$, pulling it back using the map $\pi_{X-D}:\overline{X-D}=(X-D)\underset{Spec\,\mathbb{F}_q}{\times}Spec\,\bar{\mathbb{F}}_q\rightarrow X-D$ by proposition \re{FQWS} we obtain a Weil local system $(\bar{\mathcal{F}}_{\rho},\psi_{\rho})$ on $\overline{X-D}$. Hence, by theorem \re{MTRP}, there exists a unique local system $E_{\bar{\mathcal{F}}_{\rho}}$ on $Pic_D(\bar{X})$ such that $act^*E_{\bar{\mathcal{F}}_{\rho}}\cong\bar{\mathcal{F}}_{\rho}\boxtimes E_{\bar{\mathcal{F}}_{\rho}}$. 

\begin{note}
To make notations simpler, let us denote the geometric Frobenii: $F_{\overline{X-D}},\,F_{Pic_D(\bar{X})}$ and $F_{\overline{X-D}\times Pic_D(\bar{X})}$ by $F_X,\,F_P$ and $F_{XP}$  respectively.
\end{note}

We want a structure of a Weil sheaf on $E_{\bar{\mathcal{F}}_{\rho}}$. We must establish an isomorphism $j:F^*_PE_{\bar{\mathcal{F}}_{\rho}}\xrightarrow{\approx}E_{\bar{\mathcal{F}}_{\rho}}$.
If we show that $F^*_PE_{\bar{\mathcal{F}}_{\rho}}$ also satisfies the property: $act^*F^*_PE_{\bar{\mathcal{F}}_{\rho}}\cong \bar{\mathcal{F}}_{\rho}\boxtimes F^*_PE_{\bar{\mathcal{F}}_{\rho}}$ then we obtain the isomorphism $j$ from the uniqueness assumption of the theorem.

Indeed, since the Frobenius is functorial (remark \re{FAF} (2)), it commutes with $act$, so we obtain an isomorphism $act^*F^*_PE_{\bar{\mathcal{F}}_{\rho}}\cong F_{XP}^*act^*E_{\bar{\mathcal{F}}_{\rho}}$. By the property of $E_{\bar{\mathcal{F}}_{\rho}}$ the latter is isomorphic to $F_{XP}^* (\bar{\mathcal{F}}_{\rho}\boxtimes E_{\bar{\mathcal{F}}_{\rho}})$. Again by functoriality, the Frobenius commutes with projection, so this is the same as $F_X^*\bar{\mathcal{F}}_{\rho}\boxtimes F_P^*E_{\bar{\mathcal{F}}_{\rho}}$. But since $\bar{\mathcal{F}}_{\rho}$ is a Weil sheaf we get an isomorphism to $\bar{\mathcal{F}}_{\rho}\boxtimes F_P^*E_{\bar{\mathcal{F}}_{\rho}}$. Composing all these isomorphisms we get the required isomorphism: $act^*F^*_PE_{\bar{\mathcal{F}}_{\rho}}\cong \bar{\mathcal{F}}_{\rho}\boxtimes F^*_PE_{\bar{\mathcal{F}}_{\rho}}$. Therefore $E_{\bar{\mathcal{F}}_{\rho}}$ is a Weil sheaf.

\begin{center}
$\xymatrix{
F_{XP}^*act^*E_{\bar{\mathcal{F}}_{\rho}}\ar[r]\ar[rd]\ar@/^2pc/[rrr]^{F^*_{XP}i}&
act^*F^*_PE_{\bar{\mathcal{F}}_{\rho}}\ar[r]\ar[d]&
\bar{\mathcal{F}}_{\rho}\boxtimes F_P^*E_{\bar{\mathcal{F}}_{\rho}}\ar[d]&
F_{XP}^*(\bar{\mathcal{F}}_{\rho}\boxtimes E_{\bar{\mathcal{F}}_{\rho}})\ar[l]\ar[ld]
\\
&
act^* E_{\bar{\mathcal{F}}_{\rho}}\ar[r]^{i}&
\bar{\mathcal{F}}_{\rho}\boxtimes E_{\bar{\mathcal{F}}_{\rho}}
}$
\end{center}

We claim that the isomorphism $act^* E_{\bar{\mathcal{F}}_{\rho}}\cong \bar{\mathcal{F}}_{\rho}\boxtimes E_{\bar{\mathcal{F}}_{\rho}}$ is an isomorphism of Weil sheaves. Indeed, all the maps in the above diagram are isomorphisms. The middle square commutes since $j:F^*_PE_{\bar{\mathcal{F}}_{\rho}}\xrightarrow{\approx}E_{\bar{\mathcal{F}}_{\rho}}$ is an isomorphism of Hecke eigensheaves. The curved arrow also commutes with the upper row since the Frobenius commutes with pullback and tensor product. The diagonal arrows are defined as the composition of the corresponding maps such that the triangles will also commute, therefore the diagram is commutative. But by definition of the Weil structure on the tensor product and pullback (\re{TPWS} and \re{PBWS} respectively), the diagonal arrows are exactly the action of Frobenius on $act^* E_{\bar{\mathcal{F}}_{\rho}}$ and $\bar{\mathcal{F}}_{\rho}\boxtimes E_{\bar{\mathcal{F}}_{\rho}}$. Therefore the commutativity of the big curved trapeze is exactly the statement that $i$ is an isomorphism of Weil sheaves. It will follow later from the construction, that the isomorphism $m^*E_{\mathcal{F}}\cong E_{\mathcal{F}} \boxtimes E_{\mathcal{F}}$ of the generalized Hecke eigensheaf property \re{GHEP} will also be an isomorphism of Weil sheaves.

We can now finish the proof, let us define the character $\xi$ by: $\xi=f^{E_{\mathcal{F}}}:K^{\times}\backslash\mathbb{A}_K^{\times}/\mathbb{O}^{\times}_D\cong Pic_D(\mathbb{F}_q)\rightarrow \bar{\mathbb{Q}}_{\ell}^{\times}$, we want to show that $\xi$ is indeed a homomorphism. To do this, we use the generalized Hecke eigensheaf property \re{GHEP}.  $E_{\mathcal{F}}$ has a structure of a Weil sheaf, and therefore so do $m^*E_{\mathcal{F}}$ and $E_{\mathcal{F}} \boxtimes E_{\mathcal{F}}$. The isomorphism of the Hecke eigensheaf property $m^*E_{\mathcal{F}}\cong E_{\mathcal{F}} \boxtimes E_{\mathcal{F}}$ is an isomorphism of Weil sheaves on $Pic_D\times Pic_D$. Therefore, the functions that they define should be equal, so let us calculate their value on the point $(x,y)$ where $x,y\in Pic_D(\mathbb{F}_q)$. By remark \re{RSFC},  $f^{m^*E_{\mathcal{F}}}(x,y)=f^{E_{\mathcal{F}}}(xy)=\xi(xy)$, and by the same remark $f^{E_{\mathcal{F}}\boxtimes E_{\mathcal{F}}}(x,y)=f^{E_{\mathcal{F}}}(x)f^{E_{\mathcal{F}}}(y)=\xi(x)\xi(y)$. Hence, $\xi(xy)=\xi(x)\xi(y)$.

The last thing remains to prove is that  $\rho(Fr_v)=\xi(\pi_v)$ for all primes $v\notin supp(D)$.  Let $x\in X-D$ be the point corresponds to $v$, in particular it is defined over $\mathbb{F}_q$. There exists an isomorphism of Weil sheaves on $(X-D)\times Pic_D$: $act^*E_{\mathcal{F}}\cong \mathcal{F}\boxtimes E_{\mathcal{F}}$. Hence the functions these Weil sheaves defines should be equal, so let us calculate their value on the point $(x,1)\in (X-D)(\mathbb{F}_q)\times Pic_D(\mathbb{F}_q)$ using remark \re{RSFC}.
$f^{act^*E_{\mathcal{F}}}(x,1)=f^{E_{\mathcal{F}}}(act(x,1))= f^{E_{\mathcal{F}}}(\mathcal{O}([x]))=\xi(\pi_v)$ where the last equality holds since $[x]$ corresponds to $\pi_v$ under the isomorphism between $Cl_D$ and $K^{\times}\backslash\mathbb{A}_K^{\times}/\mathbb{O}^{\times}_D$. On the other hand, by the same remark and since $\xi$ is homomorphism $f^{\pi_1^*\mathcal{F}\otimes \pi_2^*E_{\mathcal{F}}}(x,1)=f^{\mathcal{F}}(x)f^{E_{\mathcal{F}}}(1)= f^{\mathcal{F}}(x)\xi(1)=f^{\mathcal{F}}(x)$. But $f^{\mathcal{F}}(x)=\rho(Fr_v)$ by remark \re{RXSF} (2). We conclude that $\rho(Fr_v)=\xi(\pi_v)$ and we are done.
\end{proof}

\section{Proof of the main theorem - unramified case}\label{S:PMTU}

In this section we prove our main theorem in the unramified case, theorem \re{MTU}. The proof will proceed according to the following steps:

\begin{enumerate}
\item
We first construct a local system $\mathcal{F}^{(d)}$ on the symmetric product $X^{(d)}$.
\item
Using the map $\phi^{(d)}:X^{(d)}\rightarrow Pic^d(X)$, we give a local system on $Pic^d(X)$.
\item
By the eigensheaf condition, we will extend it to a local system $E_{\mathcal{F}}$ all $Pic(X)$.
\end{enumerate}

This section is divided into three subsections, devoted for each part of the proof.

\subsection{Construction of local system on the symmetric product}

We start by studying some properties of the symmetric product.

\begin{prop}
If $X,Y$ are (quasi) projective varieties, then so is $X\times Y$.
\end{prop}

\begin{proof}
Using the Segre embedding, details can be found in \cite[1.7]{Hu}.
\end{proof}

\begin{prop}
Let $X$ be a (quasi) projective variety with an action of a finite group $G$. then the quotient $X/G$ exist, and it is a (quasi) projective variety.
\end{prop}

\begin{proof}
\cite[Lecture 10]{Harr}.
\end{proof}

\begin{con}
Let $X$ be a smooth, (quasi) projective connected algebraic curve. Then according to what we just proved, the power $X^d$ is a (quasi) projective and (clearly) connected algebraic variety of dimension $d$. The symmetric group $S_d$ acts on this product by permuting the points, we get quotient which is (quasi) projective connected algebraic variety of dimension $d$, it is called the symmetric power of $X$ and denoted $X^{(d)}$.
\end{con}

We now want to show that:

\begin{prop}\label{E:SPS}
$X^{(d)}$ is nonsingular.
\end{prop}

\begin{rem}
This is not true for the higher dimensional case.
\end{rem}

To prove this we must state first some well-known results.

\begin{thm}\label{E:RNR}
A Noetherian local ring is regular iff its completion is.
\end{thm}

\begin{proof}
\cite[Proposition 11.24]{AM}.
\end{proof}

\begin{thm}[The fundamental theorem of symmetric polynomials]
The algebra of symmetric polynomials over a field $k$ in $n$ variables is isomorphic to the free algebra of rank $n$ over $k$: $k[x_1,...,x_n]\cong k[x_1,...,x_n]^{S_n}$. After a completion at the ideals generated by $x_1,...,x_n$ it follows that also $k[[x_1,...,x_n]]\cong k[[x_1,...,x_n]]^{S_n}$.
\end{thm}

\begin{proof}
\cite[Theorem 3.89]{Vi}.
\end{proof}

\begin{proof}[Proof of proposition \re{SPS}]

We can now prove the proposition, first recall that smoothness is preserved under change of the base field, therefore $X^{(d)}\rightarrow Spec\,k$ is smooth iff $X^{(d)}\underset{Spec\,k}{\times}Spec\,\bar{k}\rightarrow Spec\,\bar{k}$ so we can assume without loss of generality that $k$ is algebraically closed.

By theorem \re{RNR} it is enough to prove that the completion of the local ring at any point of $X^{(d)}$ is regular.

We check the point $x=(x_1...,x_d)$ of $X^d$. Without loss of generality assume that
\begin{center}
$y_1:=x_1=x_2=...=x_{d_1}\neq y_2:=x_{d_1+1}=...=x_{d_2-d_1}\neq...\neq y_m:=x_{d_{m-1}+1}=...=x_d$
\end{center}
otherwise we can switch between two coordinates without changing the local ring (by commutativity of the tensor product).

The completion of the local ring at a point $x_i$ is: $\hat{\mathcal{O}}_{X,x_i}\cong k[[x]]$, and in the product the completion is simply:

\begin{center}
$\hat{\mathcal{O}}_{X^d,x}=k[[x_1]]\hat{\otimes}...\hat{\otimes}k[[x_d]]\cong k[[x_1,...,x_d]]$
\end{center}

A function on the symmetric product is just a symmetric function, therefore the completion of the local ring is:

\begin{center}
$\hat{\mathcal{O}}_{X^{(d)},x}=k[[x_1,...,x_d]]/{(S_{d_1}\times...\times S_{d_m})}\cong k[[x_1,...,x_{d_1}]]/S_{d_1}\hat{\otimes}...\hat{\otimes} k[[x_{d_{m-1}+1},...,x_d]]/S_{d_m}$
\end{center}

But by the symmetric polynomial theorem the latter is isomorphic to:
\begin{center}
$k[[x_1,...,x_{d_1}]]\hat{\otimes}...\hat{\otimes} k[[x_{d_{m-1}+1},...,x_d]]\cong k[[x_1,...,x_d]]$
\end{center}
which is regular, therefore $X^{(d)}$ is nonsingular at $x$ and we are done.

\end{proof}

Let us begin now with the process of constructing a local system. From a $n$-dimensional $\ell$-adic local system $\mathcal{F}$ on $X$, we can get a $n^d$-dimensional local system $\mathcal{F}^{\boxtimes d}$ on $X^d$. Our next goal is to provide a local system $\mathcal{F}^{(d)}$ on the symmetric product $X^{(d)}$.

To do this, consider the quotient morphism $\pi:X^d\rightarrow X^{(d)}$, and define:

\begin{defn}
We define the {\em diagonal} of $X^d$ to be the set of points with at least two coordinates equal. We denote it by $diag$. By abuse of notation, its image under $\pi$ will also denoted by $diag$.
\end{defn}

$S_d$ acts freely on $X^d-diag$ (by definition $diag$ is the set of fixed points of nontrivial permutations). By \cite[IV.2.3]{Kn} ,restricting $\pi$ to the complement of the diagonal we obtain a finite Galois covering with Galois group $S_d$, $\pi|_{X^d-diag}:X^d-diag\rightarrow X^{(d)}-diag$. But $\mathcal{F}^{\boxtimes d}|_{X^d-diag}$ is a $G$-equivariant sheaf. Indeed, in the notation of the definition of $G$-equivariant sheaf \re{ESD}, to define the isomorphism $\alpha$ we must establish for each $g\in S_d$ and $f:U\rightarrow X^d$ \'etale compatible isomorphisms $\mathcal{F}^{\boxtimes d}(g^{-1}U)\xrightarrow{\approx} \mathcal{F}^{\boxtimes d}(U)$, where $g^{-1}U$ is the composition $g^{-1}\circ f:U\rightarrow X^d$ but from the commutativity of the following diagram:

\begin{center}
$\xymatrix{
g^{-1}U=U\ar@{=}[r]\ar[d]^{g^{-1}\circ f}&
U\ar[d]^{f}&
\\
X^d\ar[r]^{g}\ar[d]^{\pi_i}&
X^d\ar[d]^{\pi_{g(i)}}&
\\
X\ar@{=}[r]&
X
}$
\end{center}

We see that $(g^{-1}\circ f)^*\pi_i^*\mathcal{F}= f^*\pi_{g(i)}^*\mathcal{F}$. Hence we can define the isomorphism $\mathcal{F}^{\boxtimes d}(g^{-1}U)=\overset{d}{\underset{i=1}{\bigotimes}}(\pi_i^*\mathcal{F})(g^{-1}U) \xrightarrow{\approx} \mathcal{F}^{\boxtimes d}(U)=\overset{d}{\underset{i=1}{\bigotimes}}(\pi_i^*\mathcal{F})(U)$ to be simply the permutation of the factors.

Therefore, by proposition \re{EVLS} and its remark, we have a local system on $X^{(d)}-diag$: $\mathcal{F}^{(d)}=(\pi_*\mathcal{F}^{\boxtimes d})^{S_d}$ and $\pi^*\mathcal{F}^{(d)}=\mathcal{F}^{\boxtimes d}$.

The next step is to extend the local system to the diagonal, since local systems are equivalent to representations of the fundamental group, by theorem \re{SSC2}, it is enough to prove that the local system extends to an open subscheme $U$ of $X^{(d)}$ with $codim(X^{(d)}-U,X^{(d)})\geq2$. We take $U$ to be the set of all points where at most two coordinates are equal. Indeed, $X^{(d)}-U$ is of codimension 2, since the complement of the inverse image of $U$ in $X^d$ is a finite union of subschemes defined by two independent equations (the equations $\{(x_1,...,x_n):x_i=x_j=x_k\}$ or $\{(x_1,...,x_n):x_i=x_j,  x_k=x_l\}$). Therefore, and because the question is local, it will be sufficient to extend the local system to a point $x$ with at most two coordinates equal.

We can now assume without loss of generality that $x=(x_1,...,x_d)$ where $x_1=x_2\neq x_3...\neq x_d$, and we want to extend the local system to $x$. By the gluing property of sheaves, if we can extend the sheaf locally over an \'etale covering, we can glue the extensions to get a a globally defined sheaf, therefore the question of extending a local system is an \'etale local question and we can pass to strict Henselization. Let us denote $\tilde{X}_{x_i}=Spec\,\mathcal{O}^{sh}_{X,x_i}$ the strict Henselization at the point $x_i$. The restriction of $\mathcal{F}$ under $\tilde{X}_{x_i}\rightarrow X$ is constant since $\mathcal{F}$ is locally constant. We must show that the local system extends to $(\tilde{X}_{x_1}\times...\times \tilde{X}_{x_d})/S_d\rightarrow(X\times...\times X)/S_d$. But since only $x_1=x_2$ it can be written as $(\tilde{X}_{x_1}\times\tilde{X}_{x_2})/S_2\times \tilde{X}_{x_3}\times...\times \tilde{X}_{x_d}$ and we get a commutative diagram:
\begin{center}
$\xymatrix{
\tilde{X}_{x_1}\times ...\times \tilde{X}_{x_d}\ar[r]^<<<<{r_d}\ar[d] &
(\tilde{X}_{x_1}\times\tilde{X}_{x_2})/S_2\times \tilde{X}_{x_3}\times...\times \tilde{X}_{x_d}\ar[d]
\\
\tilde{X}_{x_1}\times \tilde{X}_{x_2}\ar[r]^<<<<<<<<<<<{r_2}&
(\tilde{X}_{x_1}\times \tilde{X}_{x_2})/S_2
}$
\end{center}

Where the horizontal maps are quotients and the vertical are projections. Let $q_i:\tilde{X}_{x_1}\times ...\times \tilde{X}_{x_d}\rightarrow \tilde{X}_{x_i}$ be the projections. The sheaf on $\tilde{X}_{x_1}\times ...\times \tilde{X}_{x_d}$ is defined as $q_1^*\mathcal{F}|_{\tilde{X}_{x_1}}\otimes...\otimes q_d^*\mathcal{F}|_{\tilde{X}_{x_d}}$. Similarly, a sheaf on $\tilde{X}_{x_1}\times \tilde{X}_{x_2}$ can be defined as $\mathcal{G}^{\boxtimes 2}=q_1^*\mathcal{F}|_{\tilde{X}_{x_1}}\otimes q_2^*\mathcal{F}|_{\tilde{X}_{x_2}}$. If we will extend $\mathcal{G}^{(2)}=({r_2}_*\mathcal{G}^{\boxtimes 2})^{S_2}$ to the diagonal, then we can define a sheaf $\mathcal{F}^{(d)}|_{\tilde{X}_{x_1}\times ...\times \tilde{X}_{x_d}}=\mathcal{G}^{(2)}\otimes q_3^*\mathcal{F}|_{\tilde{X}_{x_3}}\otimes...\otimes q_d^*\mathcal{F}|_{\tilde{X}_{x_d}}$ extending the sheaf $({r_d}_*\mathcal{F}^{\boxtimes d})^{S_d}$ to the diagonal. It is indeed an extension, since the map $r_d$ can be written as $r_2\times id$ and the inverse of $\mathcal{G}^{(2)}$ and $q_3^*\mathcal{F}|_{\tilde{X}_{x_3}}\otimes...\otimes q_d^*\mathcal{F}|_{\tilde{X}_{x_d}}$ under these maps gives $q_1^*\mathcal{F}|_{\tilde{X}_{x_1}}\otimes q_2^*\mathcal{F}|_{\tilde{X}_{x_2}}$ and $q_3^*\mathcal{F}|_{\tilde{X}_{x_3}}\otimes...\otimes q_d^*\mathcal{F}|_{\tilde{X}_{x_d}}$ respectively, so we are reduced to the case $d=2$.

In the case $d=2$, we must extend the local system to the point $(x,x)\in X^{(2)}$. This is an \'etale local statement, hence we can pass to the strict Henselization. Since $\mathcal{F}$ is locally constant, its restriction to $\tilde{X}_x$ will be constant. Therefore we can assume that $\mathcal{F}$ is constant. It implies that $\mathcal{F}^{\boxtimes 2}$ is also constant (as tensor product of restrictions of constant sheaves). Now, by the equivalence of categories \re{EVLS} $\mathcal{F}^{(2)}$ is the unique sheaf that its inverse image under $\pi$ is the constant equivariant sheaf $(\mathcal{F}^{\boxtimes 2},\alpha)$ where $\alpha:p_2^*\mathcal{F}^{\boxtimes 2}\xrightarrow{\approx}p_1^*\mathcal{F}^{\boxtimes 2}$ acts permuting the factors as explained above (notation as in the definition of equivariant sheaf, \re{ESD}). But we know such a sheaf, namely the constant sheaf on $X^{(2)}-diag$. Therefore $\mathcal{F}^{(2)}$ is the constant sheaf and a constant sheaf is clearly extendable.

The extension is unique: local systems are equivalent to representations, from proposition \re{FGSM} it follows that the map $i_*:\pi_1(X^{(d)}-diag)\rightarrow \pi_1(X^{(d)})$ is surjective. Therefore if $\rho:\pi_1(X^{(d)})\rightarrow \bar{\mathbb{Q}}^*_{\ell}$ is a representation then it is uniquely determined by the restriction $(\rho\circ i_*):\pi_1(X^{(d)}-diag)\rightarrow \bar{\mathbb{Q}}^*_{\ell}$, since $\rho(a)=(\rho\circ i_*)(i_*^{-1}(a))$.

\begin{rem}\label{E:HEPX}
If $h:X\times X^{(d)}=X^{d+1}/S_d\rightarrow X^{(d+1)}=X^{d+1}/S_{d+1}$ is the quotient map (which is  defined on the points by $h(x,D)=D+[x]$) then $h^*\mathcal{F}^{(d+1)}\cong\mathcal{F}\boxtimes\mathcal{F}^{(d)}$. Indeed, consider the diagram:

\begin{center}
$\begin{CD}
X\times X^d @>id \times\pi^{(d)}>> X\times X^{(d)}\\
@| @VhVV\\
X^{d+1} @>\pi^{(d+1)}>> X^{(d+1)}
\end{CD}$
\end{center}

From the proof that $\mathcal{F}^{(d)}$ extends to the diagonal, we can see that this extension is unique, so it is enough to prove the assertion outside the diagonal. To make notation easier, from now on in this remark, we will write simply $X^d,X^{(d)}$ for $X^d-diag,X^{(d)}-diag$. Consider the Galois covering $ id\times\pi^{(d)}:X\times X^d\rightarrow X\times X^{(d)}$, by the equivalence between sheaves on $X\times X^{(d)}$ and equivariant sheaves on $X\times X^d$, it is enough to prove that the restrictions $(id \times \pi^{(d)})^*h^*\mathcal{F}^{(d+1)}\cong(id \times\pi^{(d)})^*(\mathcal{F}\boxtimes\mathcal{F}^{(d)})$ are isomorphic as $S_d$-equivariant sheaves.

The restrictions are $\mathcal{F}^{\boxtimes (d+1)}$ and $\mathcal{F}\boxtimes\mathcal{F}^{\boxtimes d}$ which are clearly isomorphic as sheaves (we denote this natural isomorphism by $i$). But since the equivariant structure on $\mathcal{F}^{\boxtimes d}=\pi^{(d)}\mathcal{F}^{(d)}$ is permutation of the factors as explained above, we know that the equivariant structure on $(id\times\pi^{(d)})^* \mathcal{F}\boxtimes\mathcal{F}^{(d)}\mathcal{F}\boxtimes\mathcal{F}^{\boxtimes d}$ is permutation of the last $d$ factors. But the equivariant structure on $\pi^{(d+1)*}\mathcal{F}^{(d+1)}=\mathcal{F}^{\boxtimes (d+1)}$ as an $S_{d+1}$-equivariant sheaf is also permutation of the factors. But $\pi^{(d+1)}$ is the composition: $X^{d+1}\xrightarrow{id\times \pi^{(d)}}X\times X^d/S_d\xrightarrow{h}X^{d+1}/S_{d+1}  $ therefore as an $S_d$-equivariant sheaf, the equivariant structure on $\mathcal{F}^{\boxtimes (d+1)}$ is also permutation of the last $d$ factors (the action of the subgroup $S_d<S_{d+1}$). So both sheaves have the same equivariant structure and we are done.

\end{rem}

\subsection{Local system on $Pic^d(X)$}\label{SS:LSPD}

\begin{note}
Recall that $Pic^d(X)$ is the connected component in $Pic(X)$ of degree $d$ line bundles on $X$.

Caution: don't confuse with $Pic_D(X)$.
\end{note}

\begin{defn}\label{E:AJDEF}
Note that a point of $X^{(d)}$ is an unordered $d$-tuple of points of $X$, which is exactly an effective Weil divisor of degree $d$, $D\in Cl^d(X)$. So at least in the level of points, $X^{(d)}$ classifies Weil divisors of degree $d$.
Therefore, on the level of points we can define the Abel-Jacobi map $\phi^{(d)}:X^{(d)}\rightarrow Pic^d$, it is defined by $D'\mapsto \mathcal{O}(D')$ where $D'$ is any effective divisor on $X$, $D'=\underset{x}{\sum}d_x[x]$ with $\sum_xd_x=d$. The functorial language, which we are not using here, shows that this is a map of schemes as in \cite[\S3]{Kl}.
\end{defn}

We want to use the Abel-Jacobi this map to get from our local system $\mathcal{F}^{(d)}$ a local system $E_{\mathcal{F}}^d$ on $Pic^d$. So we must calculate the fiber of this map, to do this, we will use the Riemann-Roch theorem, whose statement and proof can be found in \cite[Chapter IV, Theorem 1.3]{Ha}.

As a simple corollary of the Riemann-Roch theorem we obtain that if $degD'=d\geq2g-1$ (here $g$ is the genus of $X$) then $l(D')=degD'+1-g=d+1-g$. Proposition \re{UFIBER} implies that in this case the fiber of $\phi^{(d)}$ over $\mathcal{L}$, which is the complete linear system of effective divisor linearly equivalent to $\mathcal{L}$, is the projective space $\mathbb{P}^{d-g}$.


In the case $d\geq2g-1$ the map $\phi^{(d)}$ is a fibration with fiber $\mathbb{P}^{d-g}$, therefore by the first exact sequence of homotopy (theorem \re{FHES}), we have the exact sequence:

\begin{center} $\pi_1(\mathbb{P}^{d-g})\rightarrow\pi_1(X^{(d)})\rightarrow\pi_1(Pic^d)\rightarrow1$
\end{center}
but by example \re{PSC} the projective space is simply connected. So we obtain an isomorphism ${\phi^{(d)}}_*:\pi_1(X^{(d)})\xrightarrow{\approx}\pi_1(Pic^d)$. Then for any representation of $\rho$ of $\pi_1(X^{(d)})$ there exists a representation $\rho'=\rho\circ({\phi^{(d)}}_*)^{-1}$ of $\pi_1(Pic^d)$, and we get $\rho'\circ{\phi^{(d)}}_*=\rho$. It is equivalent to state it in terms of local systems: if we take $\rho$ to be the representation associated with $\mathcal{F}^{(d)}$. Define $E_{\mathcal{F}}^d$ to be the local system associated with $\rho'$. Then the relation $\rho'\circ{\phi^{(d)}}_*=\rho$ implies that $\mathcal{F}^{(d)}={\phi^{(d)}}^*E_{\mathcal{F}}^d$.
We therefore have the required local system on $Pic^d$.

\subsection{Local system on $Pic(X)$}

We now have a local system on $\underset{d\geq 2g-1}{\bigcup}Pic^d$. Restricted to this part of $Pic$, we claim that there exists an isomorphism $i_{\mathcal{F}}:act^*E_{\mathcal{F}}\xrightarrow{\approx}\mathcal{F}\boxtimes E_{\mathcal{F}}$, this is the Hecke eigensheaf property, or if we want to be more precise and pay attention to the degrees, we claim:  $i_{\mathcal{F}}:act^*E^{d+1}_{\mathcal{F}}\xrightarrow{\approx}\mathcal{F}\boxtimes E^d_{\mathcal{F}}$.

Indeed, the statement is that $i_{\mathcal{F}}$ is an isomorphism of local systems on $X\times Pic^d$. Recall that the fibers of the map $id\times \phi^{(d)}:X\times X^{(d)}\rightarrow X\times Pic^d$ are simply connected, therefore by the first exact sequence of homotopy \re{FHES} and the correspondence between local systems and representations \re{RLSC}, the map $\mathcal{G}\mapsto (id\times \phi^{(d)})^*\mathcal{G}$ is an equivalence of categories between local systems on $X\times Pic^d$ and local systems on $X\times X^{(d)}$. So it is enough to prove that the restrictions are isomorphic: $(id\times \phi^{(d)})^*act^*E^{d+1}_{\mathcal{F}}\cong(id\times \phi^{(d)})^*(\mathcal{F}\boxtimes E^d_{\mathcal{F}})$. Consider the (clearly commutative) diagram:

\begin{center}
$\xymatrix{
X\times X^{(d)}\ar[r]^{h} \ar[d]^{id\times \phi^{(d)}}&
X^{(d+1)}\ar[d]^{\phi^{(d+1)}}
\\
X\times Pic^d\ar[r]^{act}&
Pic^{d+1}
}$
\end{center}
where $h$ is as in remark \re{HEPX} and $\phi^{(d)}$ the Abel-Jacobi map defined in \re{AJDEF}. By commutativity and by remark \re{HEPX} we indeed obtain that:

\begin{center}
$(id\times \phi^{(d)})^*act^*E^{d+1}_{\mathcal{F}}\cong h^*(\phi^{(d+1)})^*E^{d+1}_{\mathcal{F}}\cong h^*\mathcal{F}^{(d+1)}\cong \mathcal{F}\boxtimes \mathcal{F}^{(d)}\cong(id\times \phi^{(d)})^*(\mathcal{F}\boxtimes E^d_{\mathcal{F}})$
\end{center}

It only remains to extend the local system to $Pic^d$, where $d<2g-1$. We argue by descending induction.

\begin{prop}\label{E:Ind}
Assume that there exists a local system $E_{\mathcal{F}}$ defined on $\underset{d\geq N+1}{\bigcup}Pic^d$ satisfying the Hecke eigensheaf property then $E_{\mathcal{F}}$ can be extended uniquely to a local system on $\underset{d\geq N}{\bigcup}Pic^d$ which also satisfies the Hecke eigensheaf property.
\end{prop}

\begin{rem}
By the above \rss{LSPD} we have the base for the induction, a local system $E_{\mathcal{F}}$ on  $\underset{d\geq 2g-1}{\bigcup}Pic^d$ which satisfied the Hecke eigensheaf property. Hence, if we prove the induction step in above proposition, we get that the local system can be extended to $\underset{d\geq N}{\bigcup}Pic^d$ for any $N\in \mathbb{Z}$, hence it can be extended to the whole $Pic$ and we are done.
\end{rem}

\begin{proof}[Proof of proposition \re{Ind}]

To prove the proposition, we must introduce some notations. Recall that for any point $x\in X$ we have the map $act_x:Pic\rightarrow Pic$ defined in definition \re{PHIDEF} to be $act_x(\mathcal{L})=\mathcal{L}\otimes \mathcal{O}([x])$. We define the map $i_x:Pic\rightarrow X\times Pic$ by $i_x(\mathcal{L})=(x,\mathcal{L})$ and let $\pi_{1,2}$ denote the projections of $X\times Pic$ to the first and second factor respectively. We immediately get that $act_x=act\circ i_x:Pic\rightarrow Pic$, $\pi_1\circ i_x=C_x:Pic\rightarrow X$ the constant function with value $x$, and $id_{Pic}=\pi_2\circ i_x:Pic\rightarrow Pic$. Therefore, applying $i_x^*$ on $act^*E_{\mathcal{F}}^{d+1}\cong\pi_1^*\mathcal{F}\otimes \pi_2^*E_{\mathcal{F}}^d$ we obtain the isomorphism
\begin{equation}\label{Eq:zeroeq}
act_x^*E_{\mathcal{F}}^{d+1}\cong C_x^*\mathcal{F}\otimes E_{\mathcal{F}}^d\cong \mathcal{F}_x\otimes E_{\mathcal{F}}^d
\end{equation}

Recall the following well known statement from linear algebra:

\begin{prop}
Let $V$ be a vector space over some field $k$ and $V^*$ its dual space, then there exists a canonical isomorphism $V\otimes V^*\xrightarrow{\approx} End(V)$.
\end{prop}

Since $\mathcal{F}_x$ is one dimensional, $End(\mathcal{F}_x)$ is naturally isomorphic to $\bar{\mathbb{Q}}_{\ell}$. Hence, tensoring the last isomorphism with $\mathcal{F}_x^*$ we get:

\begin{equation}\label{Eq:fireq}
E_{\mathcal{F}}^d\cong\mathcal{F}_x^*\otimes act_x^*E_{\mathcal{F}}^{d+1}
\end{equation}

which is true for $d\geq N+1$. The only thing we used to get this isomorphism is the Hecke eigensheaf property. Therefore, if we want the Hecke eigensheaf property to hold also for $d=N$, this equation also must be true, so it defines $E_{\mathcal{F}}^N$ given that $E_{\mathcal{F}}^{N+1}$ is already defined, hence the uniqueness assumption of the proposition holds.

Concerning the existence, we clearly want to take equation \form{fireq} with $d=N$ and some point $x$ as the definition of $E_{\mathcal{F}}^N$.

To complete the proof of proposition \re{Ind} we must show that the Hecke eigensheaf property holds for this definition of $E_{\mathcal{F}}^N$. It is done using the Hecke eigensheaf property for $d=N+1$, we will write here a few isomorphisms and immediately explain them later:

\begin{center}
$act^*E_{\mathcal{F}}^{N+1}\cong act^*(\mathcal{F}_x^*\otimes act_x^*E_{\mathcal{F}}^{N+2})\cong
\mathcal{F}_x^*\otimes(id_X\times act_{x})^*\circ act^*(E_{\mathcal{F}}^{N+2})\cong
\mathcal{F}_x^*\otimes(id_X\times act_{x})^*(\mathcal{F}\boxtimes E_{\mathcal{F}}^{N+1})\cong
\mathcal{F}\boxtimes (\mathcal{F}_x^*\otimes act_x^*E_{\mathcal{F}}^{N+1})\cong \mathcal{F}\boxtimes E^N_{\mathcal{F}}$
\end{center}

The first and last isomorphisms come from equation \form{fireq} for $d=N+1$ and $d=N$, respectively. The second is just the statement that for $x\in X$: $act\circ (id_X\times act_x)=act_x\circ act$ which follows from the commutativity and associativity of the tensor product. The third is true because the Hecke eigensheaf property holds for $E^{N+1}_{\mathcal{F}}$ and the forth again follows from properties of tensor products.

Therefore we get the Hecke eigensheaf property for $E_{\mathcal{F}}^N$ and the proof of proposition \re{Ind}.
\end{proof}

We also get the required local system in our main theorem \re{MTU}, and it proves the existence assertion in the theorem.

In order to complete the proof of the unramified version of the main theorem, we still have to show that the local system $E_{\mathcal{F}}$ is unique.

Consider the composition:
\begin{center}
$\underset{d\,times}{\underbrace{X\times...\times X}}\xrightarrow{id\times...\times id\times \phi} \underset{d-1\,times}{\underbrace{X\times...\times X}}\times Pic^1\xrightarrow{id\times...\times id\times act}\underset{d-2\,times}{\underbrace{X\times...\times X}}\times Pic^2\rightarrow...\xrightarrow{act}Pic^d$
\end{center}
The composition of the maps is $X^d\xrightarrow{\pi^{(d)}}X^{(d)}\xrightarrow{\phi^{(d)}} Pic^d$. By the Hecke eigensheaf property $act^*E^N_{\mathcal{F}}\cong \mathcal{F}\boxtimes E^{N-1}_{\mathcal{F}}$ and since we require that $\phi^*E_{\mathcal{F}}=\mathcal{F}$ we obtain that under this composition the restriction of $E^d_{\mathcal{F}}$ is the $S_d$-equivariant sheaf $(\phi^{(d)}\circ\pi^{(d)})^*E^d_{\mathcal{F}}=\mathcal{F}^{\boxtimes d}$. Therefore, by the equivalence of proposition \re{EVLS} and its remark, $\phi^{(d)*}E^d_{\mathcal{F}}=\mathcal{F}^{(d)}$. But for $d\geq 2g-1$ the fibers of $\phi^{(d)}$ are simply connected, hence $E^d_{\mathcal{F}}$ is uniquely determined by $\mathcal{F}^{(d)}$ which is determined by $\mathcal{F}$. But from proposition \re{Ind}  $E_{\mathcal{F}}$ is uniquely determined by $E^d_{\mathcal{F}}$ for $d\geq 2g-1$. Therefore $E_{\mathcal{F}}$ is uniquely determined by $\mathcal{F}$ and we are done.

We can now prove proposition \re{GHEP}:

\begin{proof}[Proof of proposition \re{GHEP}]
We first prove the proposition for $m:Pic^1\times Pic^d\rightarrow Pic^{d+1}$ $d\geq 2g-1$ and then generalize it to $m:Pic^d\times Pic^{d'}\rightarrow Pic^{d+d'}$ for $d,d'\in \mathbb{Z}$.

Consider the commutative diagram:

\begin{center}
$\xymatrix{
X^{(d)}\times X\ar[r]^{h} \ar[d]^{\phi\times\phi^{(d)}}&
X^{(d+1)}\ar[d]^{\phi^{(d+1)}}
\\
Pic^1\times Pic^d\ar[r]^{m}&
Pic^{d+1}
}$
\end{center}
where $h$ is from remark \re{HEPX}. We have just seen that $E^d_{\mathcal{F}}$ is the unique sheaf on $Pic^d$ such that its restriction $\phi^{(d)*}E^d_{\mathcal{F}}\cong\mathcal{F}^{(d)}$. Therefore, to prove that $m^*E^{d+1}_{\mathcal{F}}\cong E^1_{\mathcal{F}}\boxtimes E^d_{\mathcal{F}}$ it is enough to prove that the restriction $(\phi\times\phi^{(d)})^*m^*E^{d+1}\cong \mathcal{F}\boxtimes \mathcal{F}^{(d)}$. Indeed, by the commutativity of the diagram the former is: $h^*\phi^{(d+1)*}E^{d+1}_{\mathcal{F}}\cong h^*\mathcal{F}^{(d+1)}\cong \mathcal{F}\boxtimes \mathcal{F}^{(d)}$, where the last equality uses remark \re{HEPX}. It concludes the case $m:Pic^1\times Pic^d\rightarrow Pic^{d+1}$ $d\in \mathbb{N}$.

Concerning the general case, it is enough to prove that if the proposition is true for $m:Pic^d\times Pic^{d'}\rightarrow Pic^{d+d'}$ then it is also true for $m:Pic^{d+n}\times Pic^{d'}\rightarrow Pic^{d+d'+n}$ and for $m:Pic^d\times Pic^{d'+n}\rightarrow Pic^{d+d'+n}$ where $n\in \mathbb{Z}$. We will prove the first result (the second proceeds in a similar manner).

If we apply equation \form{zeroeq} $n$ times we get:

\begin{equation}\label{Eq:seceq}
act_x^{n*}E_{\mathcal{F}}^{d+n} \cong \mathcal{F}^{\otimes n}_x\otimes E_{\mathcal{F}}^d
\end{equation}

Note that this equation is meaningful for any $n\in \mathbb{Z}$, for negative $n$ we take the inverses of $act_x$ and $\mathcal{F}_x$ with respect to composition and tensor product respectively. The equation is not only meaningful, but also true for negative $n$, since by application of $act^{-1*}_x$ to equation \form{fireq} we get $act_x^{-1*}E_{\mathcal{F}}^{d} \cong \mathcal{F}^*_x\otimes E_{\mathcal{F}}^{d+1}$ and in the case of negative $n$ we can apply it $-n$ times to get \form{seceq} for $n$.

We start from the isomorphism $m^*E^{d+d'}_{\mathcal{F}}\cong E^{d}_{\mathcal{F}}\boxtimes E^{d'}_{\mathcal{F}}$. Clearly $m\circ(act^{-n}_x\times id)=act^{-n}_x\circ m$, so apply $(act^{-n}_x\times id)^*$ to the last isomorphism to get $m^*act^{-n*}_xE^{d+d'}\cong  act_x^{-n*}E^{d}_{\mathcal{F}}\boxtimes E^{d'}_{\mathcal{F}}$ by equation \form{seceq} we get $m^*(\mathcal{F}^{-\otimes n}_x\otimes E^{d+d'+n})\cong  (\mathcal{F}^{-\otimes n}_x\otimes E^{d+n}_{\mathcal{F}})\boxtimes E^{d'}_{\mathcal{F}}$ but an inverse image of a constant sheaf is constant, therefore $\mathcal{F}^{-\otimes n}_x\otimes m^*E^{d+d'+n}\cong \mathcal{F}^{-\otimes n}_x\otimes (E^{d+n}_{\mathcal{F}}\boxtimes E^{d'}_{\mathcal{F}})$ tensoring with $\mathcal{F}^n_x$ we get the required $m^*E^{d+d'+n}\cong E^{d+n}_{\mathcal{F}}\boxtimes E^{d'}_{\mathcal{F}}$ and it completes the proof.
\end{proof}

\section{Proof of the main theorem - ramified case}\label{S:PMTR}

We shall now prove the ramified versions of the main theorem. Most of the proof remains exactly the same. To be more precise, in steps (1) and (3) of the proof of the unramified version, we don't use the projectiveness of $X$ at all. The only difference in the proof will therefore be in stage (2). We should distinguish here between the zero characteristic case and the case $Char(k)=p>0$.

\subsection{Surjectivity of $\phi^{(d)}$}

Surjectivity is needed in order to use the first homotopy exact sequence (theorem \re{FHES}). In the unramified case, $\phi^{(d)}:X^{(d)}\rightarrow Pic^d$ was surjective for $d$ large enough. The reason was that given an invertible sheaf $0<\mathcal{L}\in Pic$, the zero divisor of one of its global sections, is an effective divisor $D_0$ with $\phi^{(d)}(D_0)=\mathcal{L}$.

It is not so simple in the ramified case, since in proposition \re{RFIBER} we assumed that $\mathcal{L}\cong \mathcal{O}(D_0)$ where $D_0$ is effective. This assumption was necessary because we used the global section $1\in H^0(X,\mathcal{L})$, which exists only in the effective case. Hence, the method of the unramified case doesn't work in the case of ineffective $D_0$. Therefore we must find another way to prove surjectivity.

Let $(\mathcal{L},\psi)\in Pic_D^d$, it is represented by a (not necessarily effective) Weil divisor $D_0\in Div(X-D)$. We want to show that if $deg D_0$ is large enough, then there exists an effective divisor $D$-equivalent to $D_0$. In order to do this, it is enough to give a global section $f\in H^0(X,\mathcal{L})$ such that $f\equiv 1 \mod [D]$, as in proposition \re{RFIBER}. But to do this, it is enough to show that $1$ is in the image of the composition map $H^0(X,\mathcal{L})\rightarrow H^0(X,\mathcal{L}|_D)\rightarrow H^0(X,\mathcal{O}|_D)$, we will show that this composition is surjective, it is enough to show that the first map is surjective, since the second is an isomorphism.

Indeed, consider the short exact sequence of sheaves:
\begin{center}
$0\rightarrow \mathcal{L}(-[D])\rightarrow \mathcal{L}\rightarrow \mathcal{L}|_D\rightarrow 0$
\end{center}
it induces a long exact sequence of cohomology:
\begin{center}
$...\rightarrow H^0(X,\mathcal{L})\rightarrow H^0(X,\mathcal{L}|_D) \rightarrow H^1(X,\mathcal{L}(-[D]))\rightarrow...$
\end{center}
but by Riemann-Roch theorem, if we take $deg D_0> deg[D]-2g+2$, we get $H^1(X,\mathcal{L}(-[D]))=0$, therefore the map $H^0(X,\mathcal{L})\rightarrow H^0(X,\mathcal{L}|_D)$ is surjective as required.

\subsection{$Char(k)=0$}

After all this preliminary work, there remains almost nothing to do. In the unramified case, we only used the fact that the fiber of the map $X^{(d)}\rightarrow Pic^d$ is simply connected (it was the projective space) and the first exact sequence of homotopy \re{FHES}. If we consider its ramified analog $(X-D)^{(d)}\rightarrow Pic^d_D$, according to proposition \re{RFIBER} the fiber over $\mathcal{O}(D_0)\in Pic_D$ it is the affine space $\mathbb{A}^{l(D_0-[D])}$. Noting that $deg(D_0-[D])=d-deg[D]$ and using the Riemann-Roch theorem exactly as in the unramified case (but for $D_0-[D]$), gives us the required fiber in the case $d-deg[D]\geq2g-1$. In this case, the fiber is (by Riemann-Roch) $\mathbb{A}^{d-deg[D]+1-g}$. But by example \re{ASC}, the affine space is simply connected in the case $Char(k)=0$, hence we are done.

\subsection{$Char(k)=p>0$, tamely ramified case}

The difference from the $Char(k)=0$ case is that the affine space is no longer simply connected here. Given a representation $\rho:G_K\rightarrow \bar{\mathbb{Q}}_{\ell}^{\times}$ with ramification bounded by $[D]=\sum x_i$ (tamely ramified), then $\rho$ is trivial on the ramification groups $G^1_{x_i}$. By proposition \re{PHRG}, $\rho^{q-1}$ is trivial on $G^0_{x_i}$, i.e. it is unramified. Both $\rho$ and $\rho^{q-1}$ define representations of $\pi_1(X-D)$, hence also representations  $\eta,\eta^{q-1}$ of $\pi_1((X-D)^{(d)})$, as in the proof of the unramified case. Recall now proposition \re{UFIBER} and proposition \re{RFIBER} which give the fibrations in the rows of the commutative diagram:
\begin{center}
$\xymatrix{
 \mathbb{A}^{d-deg[D]+1-g}\ar[r] &
(X-D)^{(d)}\ar[r]^<<<<{\phi^{(d)}}\ar[d]&
Pic_D^d\ar[d]^{f}
\\
\mathbb{P}^{d-g}\ar[r]&
X^{(d)}\ar[r]&
Pic^d&
}$
\end{center}

Since $\rho^{q-1}$ is unramified, by unramified class field theory, it gives a representation of  $\pi_1(Pic^d)$, and in particular by composition with $f$, a representation $\mu$ of $\pi_1(Pic_D^d)$ such that $\mu\circ \phi^{(d)}=\eta^{q-1}$. We use here the first exact sequence of homotopy
(theorem \re{FHES}) to conclude that the restriction of $\eta^{q-1}$ to the fiber $\pi_1(\mathbb{A}^n)$ is trivial. But by example \re{APP} $\pi_1(\mathbb{A}^n)^{ab}$ is a pro-$p$ group and $q-1$ is prime to $p$. Hence also the restriction of $\eta$ to the fiber is trivial. Using the first exact sequence of homotopy again, $\eta$ induces a representation of $\pi_1(Pic_D^d)$, which is just what we want to prove.

\end{document}